\documentclass[10pt,a4paper]{article}

\usepackage{appendix}
\usepackage{mathscinet} 
\usepackage{amssymb}
\usepackage{makeidx}
\usepackage[leqno]{amsmath}
\usepackage{cite}
\usepackage{verbatim}
\usepackage{color}
\usepackage{multicol}
\usepackage{graphics}
\usepackage{fancyhdr}
\usepackage{marginnote}
\usepackage[conny]{fncychap}
\usepackage{amsthm}
\usepackage{amssymb}
\usepackage{upgreek}
\usepackage[latin1]{inputenc}
\usepackage[english]{babel}
\usepackage{graphicx}
\usepackage{geometry}
\ifx\UseOption\undefined
\def\UseOption{Paper}
\fi
\usepackage{optional}

\newcommand{\bo}{\partial\Omega}

\newcommand{\Div}{\operatorname{Div}}
\newcommand{\Real}{\operatorname{Re}}
\newcommand{\Imag}{\operatorname{Im}}

\newcommand{\R}{\mathbb{R}}

\newcommand{\C}{\mathbb{C}}

\newcommand{\ii}{\textrm{i\,}}
\newcommand{\g}{\mathrm{g}}
\newcommand{\dd}{\mbox{d}}
\newcommand{\inc}{\tiny \mbox{inc}}
\newcommand{\refl}{\tiny \mbox{ref}}

\theoremstyle{Theorem}
\newtheorem{teor}{Theorem}
\theoremstyle{plain}
\newtheorem{corollary}{Corollary}
\newtheorem{lemma}{Lemma}
\newtheorem{prop}{Proposition}
\theoremstyle{definition}
\newtheorem{defi}{Definition}
\theoremstyle{remark}
\newtheorem{remark}{Remark}

\newcommand{\be}{\begin{equation}}
\newcommand{\ee}{\end{equation}}
\newcommand{\bnum}{\begin{enumerate}}
\newcommand{\enum}{\end{enumerate}}
\newcommand{\beq}{\begin{eqnarray}}
\newcommand{\eeq}{\end{eqnarray}}
\newcommand{\beqno}{\begin{eqnarray*}}
\newcommand{\eeqno}{\end{eqnarray*}}

\title{Stable determination of a simple metric, a covector field and a potential from the hyperbolic Dirichlet-to-Neumann map\\ \rule{14cm}{1pt}}
\author{Carlos Montalto \\Department of Mathematics \\ Purdue University}
\date{}
\graphicspath{{./images/}}
\makeindex

\begin{document}

\fancyhead{}
\fancyhead[C]{\scriptsize \leftmark}
\fancyfoot[C]{\thepage}

\renewcommand{\thefootnote}{\fnsymbol{footnote}}

\maketitle

\begin{abstract}
Let $({\cal M},\g)$ be a compact Riemmanian manifold with non-empty boundary. Consider the second order hyperbolic initial-boundary value problem
{\footnotesize 
\[
\left\lbrace 
\begin{array}{rclcr}
(\partial_t^2 + P(x,D))u  &= &0 & \mbox{in}&   (0,T)\times {\cal M},\\
u(0,x) = \partial_t u(0,x)   &=&  0 & \mbox{for}&  x\in {\cal M},\\
u(t,x)    &=&  f(t,x)& \mbox{on} & (0,T)\times \partial {\cal M},
\end{array}
\right.
\] }
where 
{\footnotesize
\[
P(x,D) = - \frac{1}{\sqrt{\det \g}}\left( - \frac{\partial}{\partial x^j} + \ii b_j\right)\g^{ij}\sqrt{\det \g}\left( - \frac{\partial}{\partial x^i} + \ii b_i\right) + q 
\]}
is a first-order perturbation of the Laplace operator $-\triangle_\g$ on $({\cal M},\g)$. Here  $b$ and $q$ are a covector field and a potential, respectively, in ${\cal M}$.
We prove H\"older type stability estimates near generic simple Riemannian metrics for the inverse problem of recovering simultaneously $\g$, b, and q from the hyperbolic Dirichlet-to-Neumann(DN) map associated,  $f\to \partial_\nu u - \ii\langle\nu,b\rangle_\g u|_{\partial {\cal M}\times [0,T]}$ modulo a class of transformations that fixed the hyperbolic DN map.
\end{abstract}
{\center
\tableofcontents
}
\newpage

\section{Introduction}
Let $({\cal M},\g)$ be an oriented and compact Riemannian manifold with non-empty boundary $\partial {\cal M}$.  In this paper we consider the stability of the inverse problem of determining a Riemannian metric together with the lower order coefficients of the second order hyperbolic initial-boundary value problem (\ref{InitProb}), from the information that is encoded in the hyperbolic Dirichlet-to-Neumann(DN) map $\Lambda_\g$, defined in equation \eqref{DN}. We consider the following question: if two hyperbolic DN maps are close in an appropriate topology, how close are their Riemannian metrics, their covector fields and their potentials?

The physical interpretation of this problem is to use boundary measurements to find the speed and trajectory of the wave of propagation (encoded in $\g$) and additional physical properties (encoded in $b$ and $q$) inside a body. The zero initial conditions in $(\ref{InitProb})$ mean that the body is initially ``quiet'' and we use a perturbation $f$ to recovered information about its interior. This problems come naturally in many applications (e.g. geophysics). 

We start by presenting our notation and the mathematical formulation of the problem. We use the Einstein summation convention throughout the paper. Denote by $\triangle_\g$ the Laplace-Beltrami operator. We choose and arbitrary but fixed system of coordinates. In each local coordinate $\g(x) = (\g_{ij}(x))$ and 
\[
\triangle_\g u = \dfrac{1}{\sqrt{\det \g}} \dfrac{\partial}{\partial x^j} \left( \g^{ij} \sqrt{\det \g} \dfrac{\partial u}{\partial x^i} \right)  
\]  
where $(\g^{ij}) = (\g_{ij})^{-1}$. 

Any second order uniformly elliptic operator with real principal part can be written in local coordinates as 
\begin{equation} \label{OperLocalCoor}
P(x,D) = - \frac{1}{\sqrt{\det \g}}\left( - \frac{\partial}{\partial x^j} + \ii b_j\right)\g^{ij}\sqrt{\det \g}\left( - \frac{\partial}{\partial x^i} + \ii b_i\right) + q 
\end{equation}
where $b$ is a complex-valued covector field on ${\cal M}$ and $q$ a complex-valued function on ${\cal M}$. Moreover, it is self-adjoint w.r.t. $L^2 := L^2({\cal M}, dV_\g)$  if and only if $b$ and $q$ are real valued, see next section for details.

For  such  operator $P(x,D)$,  consider the problem 
\begin{equation}\label{InitProb}
\left\lbrace 
\begin{array}{rclcr}
(\partial_t^2 + P(x,D))u  &= &0 & \mbox{in}&   (0,T)\times {\cal M},\\
u(0,x) = \partial_t u(0,x)   &=&  0 & \mbox{for}&  x\in {\cal M},\\
u(t,x)    &=&  f(t,x)& \mbox{on} & (0,T)\times \partial {\cal M},
\end{array}
\right.
\end{equation}
where $f\in C_{0}^1(\R_+ \times \partial {\cal M})$. Denote by $ \nu = \nu(x)$ the outer unit conormal to $\partial {\cal M}$ at $x \in \partial {\cal M}$, normalized so that $\g^{ij}\nu_i\nu_j=1$. 
The hyperbolic DN map $\Lambda_P:H^1_0 ([0,T] \times \partial {\cal M} ) \to L^2([0,T] \times \partial {\cal M})$ is defined by 
\be 
\Lambda_{P}(f) := \left.\frac{\partial u}{\partial \nu}- \ii\langle\nu,b\rangle_gu \right|_{(0,T)\times \partial {\cal M}} = 
\left. \sum_{i=1}^{n}\nu^i \dfrac{\partial u}{\partial x^i}   - 
\ii  \nu^i b_i u \; \right|_{(0,T)\times \partial {\cal M}}, \label{DN}
\ee
where $\nu^i = \g^{ij}\nu_j$.
This DN map is invariant under the group of transformations
\begin{equation} \label{CoeficientsTransformation}
\g \mapsto \g_*:= \varphi^*\g,\hspace{1cm} b \mapsto b_*:= \varphi^*b - \mbox{d}\theta \hspace{.5cm} \mbox{ and } \hspace{.5cm} q \mapsto q_* := \varphi^* q,
\end{equation}
where $\varphi:{\cal M} \to {\cal M}$ is a diffeomorphism with $\varphi|_{\partial {\cal M}} =$Id and $\theta \in C^\infty({\cal M},\R)$ with $\theta|_{\partial {\cal M}} = 0$, see Section \ref{subsec:GaugeTransf} for details. 

The inverse problem is therefore formulated in the following way: knowing $\Lambda_P$,
\begin{enumerate}
\item Can one determine the metric $\g$, the covector field $b$ and the potential $q$ up to a transformation that acts on this coefficients by (\ref{CoeficientsTransformation})? 
\item Could we do this recovery in a stable way?
\end{enumerate} 

The problem of uniqueness was addressed in several papers. When $b=0$, the first question was solved in the Euclidean case ($\g=e$) by Rakesh and Symes \cite{MR914815} using the result of Sylvester and Uhlmann \cite{sylvester1987global}. For the non Euclidean case, the spectral analogue of the inverse boundary value problem was solved for the case $b=q=0$ by Belishev and Kurylev \cite{MR1177292} using the boundary control method introduced by Belishev \cite{MR924687, MR1474359} (see also \cite{MR1399197}). This proof uses in a very essential way a unique continuation principle proven by Tataru \cite{tataru1995unique}. Because of the latter, it is unlikely that this method would prove H\"older type of stability estimates even under geometric and topological restrictions. In \cite{MR1926923}, Kurylev and Lassas gave a positive answer to the question of uniqueness. They also uniquely recovered the damping coefficient from the hyperbolic DN-map. Their result works in more general geometric conditions than generic simple metrics. They assumed the Bardos-Lebeau-Rauch controllability condition for the non-self adjoint case. Their approach uses a geometrical formulation (see \cite{MR1484000}) of the boundary control method introduced in \cite{MR924687} and hence, it is unclear how to get stability from their approach.  The spectral analogue of this problem as well as some modifications were considered in \cite{Nachman-Sylvester-Uhlmann-1988, MR922775, MR1666879, MR2372486, shiota1985inverse, MR2193218, MR1153550, KatchalovKurylevLassas2001}.

For the problem of stability there are only partial answers and is currently addressed by many authors. Conditional type of estimates are typical for such kind of inverse problems. We assume an apriori compactness type of condition of boundedness of the $C^k$ norm of the metric, the covector field and the potential for some large $k$. This condition guarantee continuity of the inverse map without control of the modulus of continuity. Weak geometric conditions of this type were studied in \cite{MR2096795}. 

For the Euclidean case and when $b=0$, Sun \cite{MR1059582} proved continuity of the inverse problem. Also for the Euclidean case, Isakov and Sun \cite{MR1158175} proved logarithmic stability for 2 dimensional domains and H\"older stability in 3 dimensions. For the anisotropic case and when $b=q=0$, Stefanov and Uhlmann in \cite{stefanov1998stability} proved conditional H\"older stability for metrics close to the Euclidean in a $C^k$ norm, later they extended their result of H\"older stability near generic simple metris in \cite{stefanov2005stable}. For the case $b=0$, Dos Santos and Bellassoued \cite{MR2737744} proved H\"older stability for the potential $q$ when $\g_0$ is a fixed simple metric, they also recover in a stable way the conormal factor of the metric $\g_0$ within the conormal class when $q=0$. 

In this paper we prove H\"older type of conditional stability when recovering: the metric $\g$, the covector field $b$ and the potential $q$; all at the same time, under the hypothesis that the metrics are near a generic simple metric. Within the process we explicitly  recover the X-ray transform along geodesics of the covector field $b$ and the potential $q$. This work was influenced from the important breakthrough paper on boundary rigidity of Stefanov and Uhlmann \cite{stefanov2005boundary} and their later work on describing the connection of boundary rigidity with the inverse problem of recovering the metric from the hyperbolic DN-map \cite{stefanov2005stable-1}. Here we present a generalization of their ideas in combination with stability estimates for the geodesic X-ray transform that were obtained on \cite{stefanov2004stability}.


\subsection*{Main result}
\addcontentsline{toc}{subsection}{Main Result}

A Riemannian manifold $({\cal M},\g)$ is simple if $\partial {\cal M}$ is strictly convex and any two point in ${\cal M}$ can be connected by a single minimizing geodesic depending smoothly on them, see Definition \ref{SimpleMetric} for more details. Since simple manifolds are diffeomorphic to the unit ball in the Euclidean space, from now on, without loss of generality, we consider the case that ${\cal M} = \overline{\Omega}$, where $\Omega$ is diffeomorphic to a ball in the Euclidean space with smooth boundary. There exist a dense open subset ${\cal G}^k(\Omega)$  of simple metrics in $C^k(\bar{\Omega})$ for $k\gg 1$, that consists of those metrics for which the X-ray transform is s-injective and stable, see Definition \ref{GenericSimpleMetrics}. In particular, this set ${\cal G}^k(\Omega)$ contains all real analytic simple metrics in $\Omega$ and also all metrics with small enough bound on the curvature (e.g. negative curvature) \cite{stefanov2005boundary}.

The main result of this paper reads as follows. Let $P$ be the operator related to  the metric $\g$, covector field $b$ and potential $q$ as in (\ref{OperLocalCoor}), similarly let $\tilde{P}$ be the operator related to $\tilde{\g}$, $\tilde{b}$ and $\tilde{q}$. Suppose that we consider the initial boundary problem (\ref{InitProb}) for both operators. 

\begin{teor} \label{MainTheorem}
There exist $k\gg 1$ and $0 < \mu <1$, such that for any $\g_0 \in {\cal G}^k(\Omega)$ there exist $\epsilon_0>0$ such that if
\begin{align}  \label{AprioriCond}
  \begin{split} 
    ||\g - \g_0||_{C(\bar{\Omega})} + ||\tilde{\g} - \g_0||_{C(\bar{\Omega})} +  ||b - \tilde{b}||_{C(\bar{\Omega})}  < \epsilon_0 \ \\
    and \hspace{4cm}\\
    \g, \tilde{\g}, b, \tilde{b}, q, \tilde{q} \mbox{ are bounded } C^k(\bar{\Omega}),\hspace{1.5cm}
  \end{split}
\end{align}
then the following holds. For any $T>\mbox{diam}_{\g_0}(\Omega)$ and $0<\epsilon \leq T$,  there exist $C>0$ and a $C^3(\bar{\Omega})$ diffeomorphism $\varphi$ and smooth function $\theta$ as in (\ref{CoeficientsTransformation}), such that  
\begin{align}
\begin{split}
||\g - \tilde{\g}_*  ||_{C^2(\bar{\Omega})}  + ||b - \tilde{b}_*  ||_{C^1(\bar{\Omega})} + ||q - \tilde{q}_*  ||_{C(\bar{\Omega})}  &\\
& \hspace{-4cm}\leq  C || \Lambda_{P} - \Lambda_{\tilde{P}}||^\mu_{H_0^1([0,\epsilon]\times \partial \Omega) \to L^2([0,T]\times \partial \Omega)}.
\end{split}
\end{align}

\end{teor}

\begin{remark}\label{FixCoordinateForCkNorms}
All $C^k(\Omega)$ norms are related to an arbitrary but fixed choice of coordinates.
\end{remark} 

The idea of the proof is to divide the recovery in three parts: First, we prove stability at the boundary using asymptotic solutions pointing in different directions to divide the information. Second, we recover the boundary distance function from the data and use boundary rigidity estimates in \cite{stefanov2005boundary} to obtain stability for the metric. Third, we translate the problem of recovering $b$ and $q$ to an X-ray type of problem.  We can explicitly recover the X-ray transform along geodesics of $b$ and $q$ from the data. We then use estimates obtained in \cite{stefanov2004stability} to get the desired stability. In each step we must separate the information to get estimates for $\g$, $b$ and $q$ separately.

A corollary of the proof is obtained if we restrict the inverse problem to a conformal class of metrics to $\g$ (i.e. $c\g$ where $c(x)$ is the conformal factor) and the instance when the covector field $b=0$. In that case, there is no gauge invariance in the information encoded in the DN$-$map and hence  we can work with general simple metrics, i.e. we can drop the generic assumption. Boundary stability in a conformal class is obtained from  \cite{Stefanov-Uhlmann-Vasy}. We then use \cite{bernshtein1980conditions} instead of \cite{stefanov2005boundary} to prove interior stability and obtain the following result that, in particular, generalizes \cite{MR2737744}.

\begin{corollary} Let $c_0\g$ be simple. Consider  $P = \triangle_{cg} + q$ and $\tilde{P} = \triangle_{\tilde{c}\g} + \tilde{q}$. There exist $k \gg 1$, $0<\mu<1$ and $\epsilon_0 >0$ such that if
\begin{equation}  \nonumber
      ||c - c_0||_{C(\bar{\Omega})}  + ||\tilde{c} - c_0||_{C(\bar{\Omega})} < \epsilon_0  \quad \mbox{ and } \quad
    c, \tilde{c},  q, \tilde{q} \mbox{ are bounded in } C^k(\bar{\Omega}),
\end{equation}
then for any $T>\mbox{diam}_{c_0\g_0}(\Omega)$ and any  $0<\epsilon \leq T$, there exist $C>0$ such that
\begin{equation} \nonumber
||c - \tilde{c}  ||_{C^2(\bar{\Omega})}  + ||q - \tilde{q}  ||_{C(\bar{\Omega})}  \leq  C || \Lambda_{P} - \Lambda_{\tilde{P}}||^\mu_{H_0^1([0,\epsilon]\times \partial \Omega) \to L^2([0,T]\times \partial \Omega)}.
\end{equation}

\par

\end{corollary} 

Acknowledgements: I would like to thank Plamen Stefanov for suggesting the problem and to An Fu for providing his notes on boundary determination.

\section{Preliminaries} 

Let us first state some mapping properties for the initial boundary value problem (\ref{InitProb}) and some remarks about self adjointness. Any second order uniformly elliptic operator with real principal part can be written as 
\be \label{SecndOrderHyperOper}
P = - \triangle_\g + B + Q
\ee
where $B$ is a complex-valued vector field which in local coordinates has the form $B = B^j \partial/\partial x^j$ and $Q$ is a complex-valued function on $\Omega$. An straight calculation shows that $P$ is self adjoint with respect to $L^2(\Omega, dV_\g)$ if and only if
\begin{equation}\label{ConditionForSelfAdjointness}
\Real{B} = 0 \mbox{ and } \Div{B} = \ii 2\Imag{Q}.
\end{equation}
If we write (\ref{SecndOrderHyperOper}) as in (\ref{OperLocalCoor}) then we get that 
\be\label{CoeficientsTransformationsForSeflfadjointness}
B = \ii 2g^{-1}b  \mbox{ and } Q = q+ |b|_\g + \ii   \Div \g^{-1}b.
\ee
where $(\g^{-1}b)^j = \g^{ij}b_i$. We then see that by (\ref{ConditionForSelfAdjointness}) the operator (\ref{OperLocalCoor}) is self adjoint w.r.t $L^2$ if and only if $b$ and $q$ are real valued.

Following \cite{KatchalovKurylevLassas2001} and \cite{MR2193218}, we have that the direct problem (\ref{InitProb}) has the following mapping properties:   
\begin{lemma} \label{DirectProblem}
Let 
\[
F \in L^1([0,T];L^2(\Omega)) \mbox{ and } f \in H^1_0([0,T]\times \partial \Omega).
\]
Then there is a unique solution $u(t,x)$ of the problem
\be \label{InitValProbSource}
\left\lbrace 
\begin{array}{rclcr}
(\partial_t^2 + P(x,D))u  &= & F(t,x) & \mbox{in}&   (0,T)\times \Omega,\\
u(0,x) = \partial_t u(0,x)   &=&  0 & \mbox{for}&  x\in \Omega,\\
u(t,x)    &=&  f(t,x)& \mbox{on} & (0,T)\times \partial \Omega,
\end{array}
\right.
\ee
such that
\[
u(t,x) \in C([0,T];H^1(\Omega)) \cap C^1([0,T];L^2(\Omega)) 
\]
and 

\[
\max_{0\leq t\leq T} \{||u(t)||_{H^1(\Omega)} + ||u_t(t)||_{L^2(\Omega)}\} \leq c(T)  \{\int_0^T ||F(t,\cdot)||_{L^2(\Omega)} dt+ ||f||_{H^1_0([0,T]\times \partial \Omega)}\}. 
\]
Moreover $\Lambda_P f  \in L^2([0,T] \times \partial \Omega)$ and
\[
|| \Lambda_P f ||_{L^2([0,T] \times \partial \Omega)} \leq c(T)  \{\int_0^T ||F(t,\cdot)||_{L^2(\Omega)} dt+ ||f||_{H^1_0([0,T]\times \partial \Omega)}\} .
\]
\end{lemma}
To simplify notation we denote from now on, 
\[
||\cdot ||_* = ||\cdot ||_{H_0^1([0,\epsilon]\times \partial \Omega) \to L^2([0,T]\times \partial \Omega)}.
\] 
where $0< \epsilon<T$ and $T>\mbox{diam}_{\g_0}(\Omega)$. More precisely, if $\Lambda$ is the hyperbolic DN map defined in \eqref{DN}, then $||\cdot ||_*$ is defined as the supremum of $||\Lambda ||_{H_0^1([0,T]\times \partial \Omega)}$ over all $f\in H_0^1([0,\epsilon]\times \partial \Omega)$ such that $||f||_{H_0^1([0,\epsilon]\times \partial \Omega)}=1$. This is well defined since one can extend $f$ as zero for $t > \epsilon$ and this extension of $f$ will be in $H_0^1([0,T]\times\partial \Omega)$ with $f|_{t=0}=0$. Moreover, because of uniqueness for (\ref{InitValProbSource}), for any $0< \epsilon< T$ and $0 < T' \leq T$ we have  
\begin{equation} \label{IneqForDNMAp}
|| \Lambda_P ||_{H_0^1([0,\epsilon]\times \partial \Omega) \to L^2([0,T']\times \partial \Omega)} \leq || \Lambda_P ||_{H_0^1([0,T]\times \partial \Omega) \to L^2([0,T]\times \partial \Omega)}.
\end{equation}


\subsection{Simple  metrics and geodesic X-ray transform}\label{SectionGenericSimpleMetricsXRayTransform}
In this section we define the set of generic simple metrics near which we can prove stability. 
\begin{defi} \label{SimpleMetric} We say that $\g$ is simple in $\Omega$, if $\partial \Omega$ is strictly convex w.r.t $\g$, and for each $x \in \bar{\Omega}$, the exponential map $\exp_x:\exp_x^{-1}(\bar{\Omega}) \to \bar{\Omega}$ is a diffeomorphism.
\end{defi}
\begin{remark}\label{RemarkExtendSimpleMetrics}
Note that since all requirements for simplicity are open, then a small $C^k(\overline{\Omega})$ perturbation of a simple metric in $\Omega$ is also simple, so we can extend $\g$ in a strictly convex neighborhood $\Omega_1 \supset \Omega$ as a simple metric in $\Omega_1$.
\end{remark}
Consider the Hamiltonian $H_\g(x,\xi) = (\g^{ij}\xi_i\xi_j)/2$ on $T^*(\Omega)$. We denote by $(x(t), \xi(t))$ the corresponding integral curves of $H_\g$ on the energy level $H_\g = 1/2$. We use the following parametrization of those bicharacteristics. Denote
\begin{equation}
\Gamma_-(\Omega_1) := \{(z,\omega)\in T^*\Omega_1; z \in \partial \Omega_1, |\omega |_\g =1,  \langle \omega, \nu(z) \rangle_\g <0 \},
\end{equation}
where $\nu(z)$ is the outer unit conormal to $\partial\Omega.$ Introduce the measure 
\[
d\mu(z,\omega) = |\langle\omega, \nu(z)\rangle_\g|dS_z dS_\omega \hspace*{0.2cm} \mbox{ on } \Gamma_-,
\]
where $dS_z$ and $dS_\omega$ are the surface measures on $\partial\Omega$ and $\{\omega \in T^*_x\Omega_1; |\omega|_\g = 1\}$ in the metric, respectively. Define $(x(t), \xi(t)) = (x(t;z,\omega), \xi(t;z,\omega))$ to be the bicharacteristics issued from $(z,\omega) \in \Gamma_-(\Omega_1).$

For any covector field $b = b_i dx^i$ we define the geodesic $X$-ray transform $I_gb$ by
\begin{equation}\nonumber 
(I_gb)(z,\omega) = \int b_i(x(t))\xi^i(t)  dt, \hspace*{1cm} (z,\omega)\in \Gamma_-,
\end{equation}
similarly for any symmetric 2-tensor $f=f_{ij}dx^idx^j$ the geodesic X-ray transform $I_gf$, which is a linearization of the boundary rigidity problem, is defined as
\begin{equation}\nonumber
(I_gf)(z,\omega) = \int f_{ij}(x(t))\xi^i(t)\xi^j(t)  dt, \hspace*{1cm} (z,\omega)\in \Gamma_-,
\end{equation}
\noindent
where $(x(t), \xi(t)) = (x(t;z,\omega), \xi(t;z,\omega))$ as above is the maximal bicharacteristics in $\Omega_1$ issued from $(z,\omega)$ and $\dot{x}^i = \g_{ij}\xi_j = \xi^i$. 

For a vector field $v$ let d$^s$ be the symmetric differential defined by $(\mbox{d}^sv)_{ij} = \frac{1}{2}(\nabla_iv_j + \nabla_iv_j)$ where $\nabla_i $ are the covariant derivatives. It is known that 
\[
I_\g\mbox{d}^sv =0 \mbox{ for any vector field } v \mbox{ with } v|_{\partial\Omega} = 0.
\]

\begin{defi} \nonumber 
We say that $I_\g$ is $s$-inyective in $\Omega$, if $Igf=0$ and $f \in L^2(\Omega)$ imply $f=\mbox{d}^sv$ for some vector field $v \in H_0^1(\Omega)$.
\end{defi}

\begin{defi}\label{GenericSimpleMetrics}
 Given $k\geq 2$, define ${\cal G}^k := {\cal G}^k(\Omega)$ as the set of all simple $C^k(\overline{\Omega})$ metrics  in $\Omega$ for with the map $I_\g$ is $s$-injective.
\end{defi}

By \cite{stefanov2005boundary}, for $k \gg 1$, ${\cal G}^k$ is an open and dense subset of all simple $C^k(\overline{\Omega})$ metrics  in $\Omega$, in particular, all real analytic simple metrics belong to ${\cal G}^k$. Also all metrics with a small enough bound on the curvature (in particular all negatively curved metrics) belong to ${\cal G}^k$, see \cite{sharafutdinov1994integral} and references in there.


\subsection{Invariance of the hyperbolic DN-map} \label{subsec:GaugeTransf}
We will consider the type of transformations that do not change the hyperbolic DN map. Let
\[
\varphi: \Omega \to \Omega
\]
be a diffeomorphism with $\varphi|_{\partial \Omega} =$Id, let us denote $x=\varphi(y)$ then for any $ v \in C_0^\infty  (\Omega)$
\[
(P(x,D)u,v)_{L^2} = (\varphi^*P(y,D) \varphi^*u,\varphi^*v)_{L^2}
\]
where the operators $\varphi^*P$ is as in (\ref{OperLocalCoor}) with 
\begin{equation}\label{DiffTrans}
\varphi^*\g, \varphi^*b \mbox{ and } \varphi^*q \hspace*{.6cm}\mbox{ instead of } \hspace*{.6cm}\g, b \mbox{ and } q
\end{equation}
respectively. Here $\varphi^*$ denotes the pull-back with respect to the metric $\g$. Since $v \in C_0^\infty(\Omega)$ is arbitrary we get 
\begin{equation} \label{ChangeVariablesInOperator}
P(x,D)u = 0 \iff \varphi^*P(y,D)\varphi^*u =0. 
\end{equation}

Making a change of variable in the hyperbolic DN map operator we see that for any $f\in C_0^1(\R_+ \times \partial \Omega)$ 
	\begin{align*}
\Lambda_P(f) &= \left. \nu^i \tfrac{\partial u}{\partial x^i} - \ii  \nu^i b_i u \; \right|_{(0,T)\times \partial \Omega} = \left. \nu^i \tfrac{\partial y^k}{\partial x^i} \tfrac{\partial \varphi^*u}{\partial y^k} - \ii  \nu^i \tfrac{\partial y^k}{\partial x^i} b_i \tfrac{\partial x^i}{\partial y^k} \varphi^*u \; \right|_{(0,T)\times \partial \Omega} \\
	& = \left. (\varphi^*\nu)^k \tfrac{\partial \varphi^*u}{\partial y^k} - \ii  (\varphi^*\nu)^k (\varphi^*b)_k \varphi^*u \; \right|_{(0,T)\times \partial \Omega} = \Lambda_{\varphi^*P}(f),
	\end{align*}
where the last equality follows from (\ref{ChangeVariablesInOperator}) and the fact that $\varphi|_{\partial \Omega} =$Id, hence $\Lambda_{\varphi^*P} = \Lambda_{P}$.

There is also another type of transformation that keep the DN map invariant, let 
	\begin{equation}\label{GaugeTransformationCovector}
\theta \in C^{\infty}(\Omega;\R), \theta|_{\partial \Omega} = 0.
	\end{equation}
We denote $P_{\theta}$ the operator as in  (\ref{OperLocalCoor}) where $b$ is replaced by 
	\begin{equation}\label{CovectorTrans}
	b_\theta = b - \mbox{d}\theta
	\end{equation}
we claim that $\Lambda_{P_\theta} =  \Lambda_{P}$. 
Similarly as before we can show that
\begin{equation} \label{OperatorGaugeTransf}
Pu =0 \iff P_\theta v = 0  
\end{equation}
where $u=e^{\ii \theta}v$, then we get that for any $f\in C_0^1(\R_+ \times \partial \Omega)$ since $e^{\ii \theta}|_{\partial \Omega} = 1$
\begin{align*}
\Lambda_P(f) &= \left. \nu^i \tfrac{\partial u}{\partial x^i} - \ii  \nu^i b_i u \; \right|_{(0,T)\times \partial \Omega}  = \left. \nu^i \tfrac{\partial e^{\ii \theta}v}{\partial x^i} - \ii  \nu^i b_i e^{\ii \theta}v \; \right|_{(0,T)\times \partial \Omega}  \\
	&=   \left. \nu^i \tfrac{\partial v}{\partial x^i} - \ii  \nu^i (b_i - \tfrac{\partial \theta}{\partial x^i}) v \; \right|_{(0,T)\times \partial \Omega} = \Lambda_{P_\theta}(f) 
\end{align*}
were the last equation follows by (\ref{OperatorGaugeTransf}). It is worth mention that transformations (\ref{DiffTrans}) and (\ref{CovectorTrans}) commute since $\varphi^*b - \varphi^*\mbox{d}\theta = \varphi^*b -  \mbox{d}\varphi^*\theta$.
Previous discussion justifies the definition of the set of operators, described by their coefficients,  that fix $\Lambda_P$ as:
\begin{equation} \label{GroupOfTransformationsByCoeficienta}
[(\g,b,q)] := \left\lbrace  (\varphi^*\g, \varphi^*b - \mbox{d}\theta, \varphi^* q):\begin{subarray}{l} \varphi \mbox{ \footnotesize is a diffeomorphism, } \varphi|_{\partial \Omega} = 1 \\  \theta \in C^{\infty}(\Omega;\C) \mbox{ \footnotesize and } \theta|_{\partial \Omega} = 0 \end{subarray}  \right\rbrace 
\end{equation}
We emphasize that the recovery of $\g$, $b$ and $q$ is done up to this class of transformations.


\subsection{Boundary normal coordinates}
In this section we will construct a diffeomorphism $\varphi$ that fixes the boundary, and a smooth function $\theta$ like in \eqref{GaugeTransformationCovector} to modify the covector field. We use the invariance of the hyperbolic DN map to modify $\g$, $b$ and $q$ near the boundary to have convenient computational properties. We sate the following well known proposition about boundary normal coordinates.
\begin{prop}\label{BoundaryNormalCoordinates}
Let $(\Omega, \g)$ be a Riemannian manifold with compact boundary $\partial \Omega$. There exist  $T>0$ and a neighborhood $N \subset \Omega$ of the boundary $\partial \Omega$ together with a diffeomorphism $\varphi : \partial \Omega \times  [0,T)     \to N$ such that: (i) $\varphi(x', 0) = x'$ for all $x' \in \partial \Omega$ and (ii) the unique unit-speed geodesic normal to $\partial \Omega$ starting at any $x_0 \in \partial \Omega$ is given by $x^n \mapsto \varphi(x_0, x^n)$. Moreover, by choosing a parametrization for the boundary near $x_0 \in \partial \Omega$, we can chose coordinates  $x = (x', x^n)$ near $x_0$, such that the line element $\dd s^2$ in the metric $\g$ is given by
\begin{equation}\nonumber
\dd s^s =   g_{\alpha\beta}\dd x'_\alpha x'_\beta + \dd x^n. 
\end{equation}  
where $\alpha, \beta$ vary from $1$ to $n-1$. In particular $x^n = \dd(x, \partial \Omega)$.
\end{prop}

For the two metrics $\g$ and $\tilde{\g}$, there exist $\varphi_1$ and $\varphi_2$ diffeomorphisms like in Proposition \ref{BoundaryNormalCoordinates}. Then 
\begin{equation}\label{Varphi}
\varphi := \varphi_1 \circ \varphi_2^{-1}
\end{equation} 
is a diffeomorphism near $\partial\Omega$ fixing $\partial\Omega$, and mapping the unit speed geodesics for $\g$ normal to $\partial\Omega$ into unit speed geodesics for $\tilde{\g}$ normal to $\partial\Omega$. By \cite{MR0117741} this diffeomorphism can be extended to a global diffeomorphism, let us called $\varphi$ again for its extension. 

\begin{figure}[h] 
 \centering
  \includegraphics[height=2.7in,keepaspectratio]{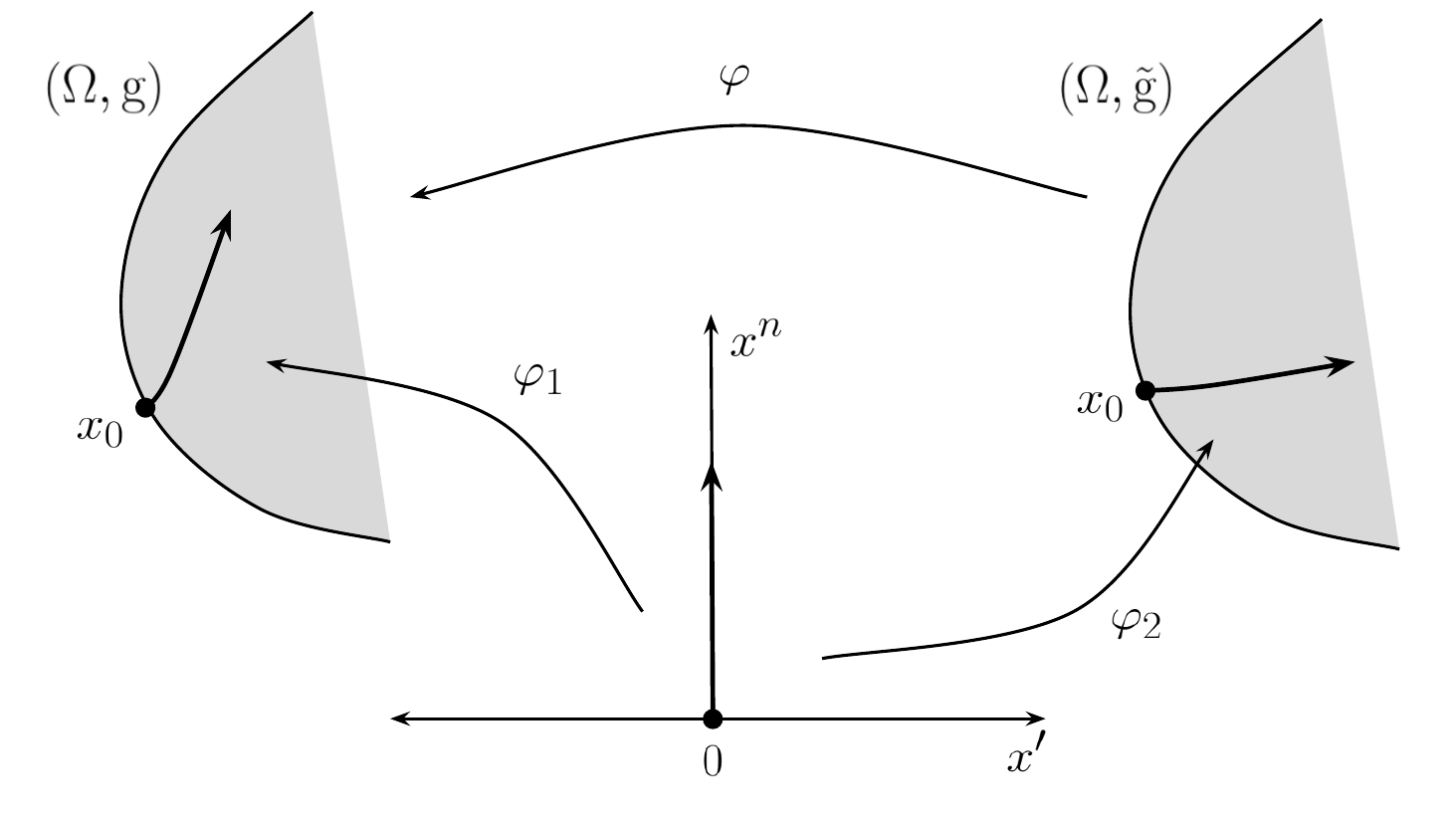}
  \caption{Boundary normal coordinates}
  \label{fig:BoundNormCoor}
\end{figure}

Notice that $\g$ and $\varphi^*\tilde{\g}$ have common normal geodesics to $\partial\Omega$, close to $\partial\Omega$, and moreover, if $(x',x^n)$ are boundary normal coordinates near a fixed boundary point for one of those metrics, they are also boundary
normal coordinates for the other metric, see Figure \ref{fig:BoundNormCoor}.

Now let $b$, $\tilde{b}$ be two covector fields in $\Omega$ as in Theorem (\ref{MainTheorem}). Use boundary normal coordinates $(x',x^n)$ near the boundary $\partial \Omega$ and let $b = b_idx^i$ and $\tilde{b}=\tilde{b}_idx^i$. Define 
\begin{align}\label{Theta}
\theta(x',x^n) & = \int_0^{x^n} (\varphi^*\tilde{b}_n- b_n)(x',t) dt
\end{align}
Extend $\theta$ so that $\theta \in C^\infty(\Omega;\C)$, notice $\theta|_{\partial \Omega} = 0$. If we use this $\theta$ as in (\ref{CoeficientsTransformation}), we get that near the boundary $\partial \Omega$
\begin{equation}\label{bnChange}
(\tilde{b}_*)_n  = (\varphi^*\tilde{b} - \mbox{d}\theta)_n = b_n.
\end{equation}
Now by the invariance of the hyperbolic DN map under this type of transformations from now on we  modify the initial coefficients $\tilde{g}$, $\tilde{b}$ and $\tilde{q}$ in the following way:
\begin{equation}
\tilde{\g} \mapsto \varphi^*\tilde{\g}, \hspace{1cm} \tilde{b}\mapsto \varphi^*\tilde{b} - \mbox{d}\theta \hspace{.5cm} \mbox{and}  \hspace{.5cm} \tilde{q} \mapsto \varphi^*\tilde{q}.
\end{equation} 

\begin{remark}[Transformation of coefficients] \label{TransformationDoesnotChangeAssumptions}
By construction of $\varphi = \mbox{Id} + O(\varepsilon_0)$ in $C^{k-2}$; therefore, the metric $\varphi_*\tilde{\g}$ also satisfies hypothesis (\ref{AprioriCond}) on Theorem \ref{MainTheorem}, with $k$ replaced by $k+3$ and for some $\varepsilon_0'>0$, such that $\varepsilon_0 \to 0$ as $\varepsilon_0' \to 0$. The same follow for $\varphi^*\tilde{b} - \mbox{d}\theta$ and $\varphi^*q$. Notice that here we are using the fact that $b$ is close to $\tilde{b}$ by hypothesis (\ref{AprioriCond}) on Theorem \ref{MainTheorem}. Hence, without loss of generality we denote $\varphi^*\tilde{\g}$, $\varphi^*\tilde{b} - \mbox{d}\theta$ and $\varphi^*\tilde{q}$  again by $\tilde{\g}$, $\tilde{b}$ and $\tilde{q}$. From now on we will follow this notation.
\end{remark}


\subsection{Interpolation Estimates}
This interpolation result is needed to use the apriori conditions in (\ref{AprioriCond}). We fix a simple metric $\g_0 \in C^k$ for $k \gg 1$. By remark \ref{RemarkExtendSimpleMetrics} extend $\g_0$ as a simple metric in some $\Omega_1 \supset \overline{\Omega}$. Let $\g, \tilde{\g}$ satisfy condition \eqref{AprioriCond} in Theorem \ref{MainTheorem}, then there exist $A>0$ and $\varepsilon_0 \ll 1$ such that
\be
\label{MetricEstimates}
||\g||_{C^k(\overline{\Omega})} + ||\tilde{\g}||_{C^k(\overline{\Omega})} \leq A, \hspace{0,5cm} ||\g-\g_0||_{C(\overline{\Omega})} +  ||\tilde{\g}-\g_0||_{C(\overline{\Omega})} \leq \varepsilon_0.
\ee
The first condition above is a typical compactness condition. We use standard interpolation estimate in Section 4.3.1 in \cite{MR1328645} to normalize the use of norm estimates. As  example, by using interpolation estimates we obtain
\be \label{interpolation}
||f||_{C^t(\overline{\Omega})} \leq C ||f||^{1-\theta}_{C^{t_1}(\overline{\Omega})}||f||^\theta_{C^{t_2}(\overline{\Omega})}
\ee
where $0 < \theta < 1$, $t_1 \geq 0$, $t_2\geq 0$, then  
\be
||\g-\g_0||_{C^t(\overline{\Omega})} + ||\tilde{\g}-\g_0||_{C^t(\overline{\Omega})} \leq C_\Omega\varepsilon_0^{(k-t)/k}
\ee
for each $t \geq 0$, if $k>t$. For our purposes, it is enough to apply (\ref{interpolation})  with $t, t_1$ and $t_2$ integers only. Estimates like (\ref{interpolation}) extend to compact manifolds with or without boundary.

\section{Boundary stability}

We will prove first stability at the boundary following \cite{stefanov2005stable}. We consider a highly oscillating solution of (\ref{InitProb}) asymptotically supported near a single geodesic transversal to $\partial \Omega$. We only need to work locally near a fixed point $x_0\in \partial \Omega$ and then use compactness of the boundary to glue the estimates. 


\subsection*{Geometrical optics solutions at the boundary}
\addcontentsline{toc}{subsection}{Geometrical optics solutions at the boundary}

Let $(x',x^n)$ be the boundary normal coordinates near $x_0$. Consider parameters $\lambda \in \R$, $\lambda \gg 1$ and $\omega' \in \R^{n-1}$. Let $\varepsilon$ and $T'$ such that $0< \varepsilon < T' \leq T$. Fix $t_0$ such that $0 < t_0 < \varepsilon$, and let $\chi\in C_0^\infty(\R_+\times \partial \Omega)$ be supported in a small enough neighborhood of $(t_0,x_0)$ of radius $\varepsilon' < \min(t_0, \varepsilon - t_0)$ and equals to 1 in a smaller neighborhood of $(t_0, x_0)$. We define $u$ to be the solution of (\ref{InitProb}) with $T'$ instead of $T$ and
\be \label{InitCond}
f= e^{\ii\lambda(t-x'\cdot \omega')}\chi(t,x').
\ee
One can get an asymtotic expansion for $u$ in a neighborhood $[0,T'] \times U \subset \R^+ \times \Omega$ of  $(t_0,x_0)$ by looking for $u$ of the form
\be \label{GeomOptSol}
u = e^{\ii\lambda(t-\phi(x,\omega))} \sum_{j=0}^{N}\lambda^{-j}A_j(t,x,\omega) + O(\lambda^{-N-1}) \quad\mbox{in}\quad C^1([0,T'];L^2(U)),
\ee
where $N \gg 0$ is fixed and $\omega =(\omega', \omega^n)$ is such that $\g^{ij}(x_0)\omega_i\omega_j = 1$ and $\g^{ij}(x_0) \nu_i(x_0)\omega_j < 0$. In $U$ the phase function solves the eikonal equation
\be \label{EikonalEquat}
\sum_{i,j=1}^n \g^{ij}\dfrac{\partial \phi}{\partial x^i} \dfrac{\partial \phi}{\partial x^j} = 1, \hspace{0.5cm} \phi|_{\partial \Omega} = x'\cdot \omega'.
\ee
With the extra condition $\frac{\partial\phi}{\partial \nu}|_{\partial \Omega} < 0$, \eqref{EikonalEquat} is uniquely solvable near $x_0$. In our coordinates, the metric $\g$ satisfies $\g_{in} = \g^{in} = \delta_{in}$ for $i=1,\ldots,n$ and $\partial/\partial\nu = - \partial/\partial x^n$. Notice that (\ref{EikonalEquat}) determines
\be \label{DefOmegaN}
\omega_n(x) = \frac{\partial \phi}{\partial x^n} (x) >0, \forall x \in U \cap \partial\Omega.
\ee
In $[0,T']\times U$ the principal part $A_0$ of the amplitude solves the first transport equation.
\be \label{PrincPart}
LA_0 = 0, \hspace{0.5cm} A_0|_{x^n =0} = \chi
\ee
and the lower order terms solve the transport equation
\be \label{TransEquat}
\ii LA_j = -(\partial_t^2 +P)A_{j-1}, \hspace{0.5cm} A_j|_{x^n=0}=0, \hspace{0.5cm} j\geq 1, 
\ee
where 
\be\label{TransEquatGeneralForm}
L = 2\partial_t + 2\sum_{i,j=1}^{n}\g^{ij}\dfrac{\partial \phi}{\partial x^j}\dfrac{\partial}{\partial x^i} - 2\ii \sum_{i,j=1}^n \g^{ij}b_i \dfrac{\partial \phi}{\partial x^j} + \Delta_\g\phi.
\ee

The construction of the solution $u$ guarantees that $A_j$, $j=1,\ldots, N$ are supported in a small neighborhood, depending on the size of $\sup \chi$, of the characteristics issued from $(t_0,x_0)$ in the codirection $(1,\omega)$. By the way we choose $\varepsilon'$, the term
\[
u_N := e^{\ii\lambda(t-\phi)}\sum_{j=0}^{N}\lambda^{-j}A_j
\]
in (\ref{GeomOptSol}) satisfies the zero initial condition in (\ref{InitProb}). Moreover, $u_N$ satisfies the boundary condition $u_N=f$ with $f$ as in (\ref{InitCond}), provided that $T'$ is such that $0 < T' -t_0$, is small enough so that the wave does not meet $\partial \Omega$ again (if it does, we need to reflect it off the boundary, as in next section). Write 
\[
u = u_N + w.
\]

Then $w = w_t = 0$ for $t=0$, and $w|_{(0,T')\times \partial \Omega} =0$ and $(\partial_{t}^2 + P)w = e^{\ii \lambda(t-\phi)}\lambda^{-N}(\partial_t^2 + P)A_N$. By Lemma \ref{DirectProblem} we obtain (\ref{GeomOptSol}) with the estimate remainder in $C^1([0,T'];L^2(U))$. 


\subsection*{Stability of higher order derivatives at the boundary}
\addcontentsline{toc}{subsection}{Stability of higher order derivatives at the boundary}

In order to get stability at the boundary we are going to take advantage of the freedom that we have of choosing the initial co-direction $\omega(x)$ of the solution $u$ constructed in previous section. For that we need the following lemma that allows to use finitely many co-direction to recover stability estimates for the metric, the covector field and the potential. We postpone its proof to the end of this section.

\begin{lemma}\label{lemma2}
Let $(\Omega, \g)$ be a Riemannian $n$-compact manifold with continuous metric. Let $r >0$ and $({\cal V}, \psi)$ any chart containing $x_0 \in \Omega$. Consider the  functional
\[
L(\omega, h, b) = h^{ij}(x)\omega_i(x)\omega_j(x) + b^i(x)\omega_i(x),
\]
where $h = h^{ij}\partial_i\otimes\partial_j$ is a symmetric 2-vector field, $b = b^i \partial_i$ is a 1-vector field and $\omega(x) \in S^r_x (T\Omega) = \{\omega \in T_x(T\Omega): |\omega|_{\g(x)} = r \} $. There exist $V\subset {\cal V}$ open neighborhood of $x_0$ and $N = n(n+3)/2$ co-directions $\omega^1(x), \ldots, \omega^N(x) \in S^r_x (T\Omega),$ such that 
\[
\sum_{i,j}^n |h^{ij}(x)|^2 + |b^i(x)|^2 \leq C \max_{k = 1, \ldots, N}|L(\omega(x),h(x),b(x))|^2 \quad \mbox{for}  \quad x \in V,
\]
and $C>0$ independent of $h$, $b$ and $\omega^k$ for $k = 1, \ldots, k$.
\end{lemma}

The main result of this section is the following:
\begin{teor} \label{BoundaryStability} For any $\mu <1$, $m \geq 0$, there exist $k \gg 1$, such that for any $A >0$, if
\[
\g, \tilde{\g}, b, \tilde{b}, q, \tilde{q} \mbox{ are bounded in } C^k(\bar{\Omega}) \mbox{ by } A, 
\]
then there $\exists C>0$ depending in $A$ and $\Omega$, such that
\begin{description}
\item{(i)} $\sup_{x\in \partial \Omega,|\gamma| \leq m} |\partial^\gamma (\g - \varphi^*\tilde{\g} ) |\leq C ||\Lambda_P - \Lambda_{\tilde{P}} ||_*^{\mu/2^m} $

\item{(ii)}  $\sup_{x\in \partial \Omega,|\gamma| \leq m} |\partial^\gamma (b - (\varphi^*\tilde{b} - \mbox{d}\theta)) |\leq C||\Lambda_P - \Lambda_{\tilde{P}} ||_*^{\mu/2^{m+1}}$

\item{(iii)} $\sup_{x\in \partial \Omega,|\gamma| \leq m} |\partial^\gamma (q - \varphi^*\tilde{q} ) |\leq C||\Lambda_P - \Lambda_{\tilde{P}}||_*^{\mu/2^{m+2}}$
\end{description}
where $\varphi$  and $\theta$ are as in (\ref{Varphi}) and (\ref{Theta}), respectively. $C^k$ norms are taken as in Remark \ref{FixCoordinateForCkNorms}.
\end{teor}
\begin{proof}

By Remark \ref{TransformationDoesnotChangeAssumptions}, without loss of generality, we denote $\varphi^*\tilde{\g}$, $\varphi^*\tilde{b} - \mbox{d}\theta$ and $\varphi^*\tilde{q}$ again  by $\tilde{\g}$, $\tilde{b}$ and $\tilde{q}$. To simplify notation let $\delta = ||\Lambda_P - \Lambda_{\tilde{P}}||_*$. In this proof we will denote $C$ by various constants depending only on $\Omega, A$ and the choice of $k\gg 1$ in Theorem \ref{MainTheorem}.  Using local boundary normal coordinates as before, let $I \times V \subset \R_+ \times \partial\Omega$ be a neighborhood  of $t=t_0$, $(x',x^n) = (0,0)$, where $\chi =1$. Since $\frac{\partial}{\partial \nu} = - \frac{\partial}{\partial x^n}$ and by Lemma \ref{DirectProblem}, we have
\begin{align} \label{DNMapBoundary}
\Lambda_P f &= e^{i\lambda(t-x\cdot \omega)} \left(i\lambda \frac{\partial \phi}{\partial x^n} - \sum_{j=0}^{N} \lambda^{-j} \left(\frac{\partial }{\partial x^n}  - ib_n\right)A_j \right)  + O(\lambda ^{-N-1}) 
\end{align}
in $L^2(I\times V)$; similarly for $\Lambda_{\tilde{P}} f$. Then
\begin{align*}
\frac{\partial \phi}{\partial x^n}-\frac{\partial \tilde{\phi}}{\partial x^n}  & = \frac{1}{i\lambda}e^{-i\lambda(t-x\cdot \omega)} (\Lambda_P f - \Lambda_{\tilde{P}}f     \\
 & \;\;\; +  \frac{1}{i\lambda} \left(\sum_{j=0}^{N}\lambda^{-j} \left( \frac{\partial A_j}{\partial x^n}- \frac{\partial\tilde{A_j}}{\partial x^n}\right)  -i\lambda^{-j} (b_jA_j - \tilde{b_j}\tilde{A_j})   \right) + O(\lambda^{-N-2})
\end{align*}
in $L^2(I\times V)$. Hence,
\begin{align*}
\left|\left| \frac{\partial \phi}{\partial x^n}-\frac{\partial \tilde{\phi}}{\partial x^n}\right|\right|_{L^2(V)}  & \leq \frac{C}{\lambda} \left(\delta ||f||_{H^1([0,\varepsilon]\times \partial \Omega)}  \right) + \frac{C}{\lambda}
\end{align*}
Notice that  $||f||_{H^1([0,\varepsilon]\times \partial \Omega)} \leq C\lambda$. Now take $\lambda \to \infty$ above to get
\be\label{PhaseEstimate}
\left|\left|\frac{\partial \phi}{\partial x^n} - \frac{\partial \tilde{\phi}}{\partial x^n} \right|\right|_{L^2(V)} \leq C\delta.
\ee
By the eikonal equation (\ref{EikonalEquat}), in $V\subset \partial \Omega$, we have 
\be \label{EikonalBoundary}
\omega_n(x) = \frac{\partial \phi}{\partial x^n} = \left(1-\sum_{\alpha,\beta =1}^{n-1}\g^{\alpha\beta}(x)\omega_\alpha\omega_\beta \right)^{1/2}, 
\ee
and similarly for $\tilde{\omega}_n = \partial \tilde{\phi}/\partial x^n$. Since $ ||(\omega_n)^2 - (\tilde{\omega}_n)^2||_{L^2(V)}  = O(\delta)$ by (\ref{PhaseEstimate}). Then we get
\[
||(\g^{\alpha\beta}(x) - \tilde{\g}^{\alpha\beta}(x)) \omega_\alpha \omega_\beta||_{L^2(V)} =  O(\delta). 
\]
and by Lemma \ref{lemma2} we have
\be\label{L2LocalEst1}
||\g-\tilde{\g}||_{L^2(V)}\leq C\delta.
\ee
Using compactness the manifold we extend (\ref{L2LocalEst1}) to the whole $\bo$ hence:
\be\label{Estimates1L2}
||\g-\tilde{\g}||_{L^2(\partial \Omega)} \leq C\delta. 
\ee
\noindent
We use interpolation estimates in Sobolev spaces and Sobolev embedding theorems, to get that for any $m\ge0$ and $\mu<1$
\be\label{Estimates1}
\|\g-\tilde \g\|_{C^m(\bo)}  \leq C\delta^\mu,
\ee
provided that $k\gg1$.

To estimate the difference of the first normal derivatives of $\g$ and $\tilde \g$ and the difference between $b$ and $\tilde{b}$ we use up to the principal part in (\ref{DNMapBoundary}), as before 
\begin{align*}
\frac{\partial A_0}{\partial x^n} - \frac{ \partial \tilde A_0}{\partial x^n} - i(b_nA_0-\tilde b_n \tilde A_0) &=-e^{-i\lambda (t-  x\cdot \omega )}(\Lambda_P(u)-\Lambda_{\tilde{P}} (\tilde u))+i\lambda \left(\frac{\partial \phi}{\partial x^n}-\frac{\partial \tilde \phi}{\partial x^n}\right) \\
& \hspace*{-2cm} -\sum_{j=1}^N \lambda^{-j}\left(\frac{\partial A_j}{\partial x^n}- \frac{\partial \tilde A_j}{\partial x^n} \right)+\sum_{j=1}^N i\lambda^{-j}(b_nA_j-\tilde b_n \tilde A_j)+O(\lambda^{-N-1})
\end{align*}
to obtain
\be
\Big\| \frac{\partial A_0}{\partial x^n} - \frac{ \partial \tilde A_0}{\partial x^n}-i(b_nA_0-\tilde b_n \tilde A_0)\Big\|_{L^2(V)} \le C\left( \lambda\delta + \delta + \lambda^{-1} \right).
\ee
The r.h.s above is minimized when $\lambda = \delta^{-1/2}$
\be \label{BoundPrincPart}
\Big\| \frac{\partial A_0}{\partial x^n} - \frac{ \partial \tilde A_0}{\partial x^n}-i(b_nA_0-\tilde b_n \tilde A_0)\Big\|_{L^2(V)} \leq C\delta^{1/2}.
\ee
The transport equation on $(t_0 - \varepsilon_1,t_0+\varepsilon_1)\times V$ implies
\[
2\omega_n\frac{\partial A_0}{\partial x^n}+\Delta_\g\phi-2i\sum_{i,j=0}^n \g^{ij}b_i\omega_j=0.
\]
Assuming the $\omega_n , \tilde{\omega}_n > \eta >0$ for small $\eta$ then by (\ref{EikonalEquat}) and since $\g^{in} = \delta_{in}$,
\beqno
2\omega_n\frac{\partial A_0}{\partial x^n} - 2i\omega_nb_n & = & \frac{\omega_n}{2 \det{\g}} 
\frac{\partial \det{\g}}{\partial x^n} + \frac{\partial^2 \phi}{\partial^2 x^n} -2i b^\beta \omega_\beta + R \\ 
 & = &  \frac{\omega_n}{2 \det{\g}} \frac{\partial \det{\g}}{\partial x^n} + \frac{1}{2\omega_n}\frac{\partial \g^{\alpha\beta}}{\partial x^n}\omega_\alpha \omega_\beta -2i b^\beta \omega_\beta + R,
\eeqno
where $R$ involves tangential derivatives of $\g$ that we can estimate by (\ref{Estimates1L2}) and $\partial \phi/\partial x^n$ that we can estimate by (\ref{PhaseEstimate}). Therefore by (\ref{BoundPrincPart}),
\begin{align}
\frac{\omega_n}{2 \det{\g}} 
\frac{\partial \det{\g}}{\partial x^n} - \frac{\tilde{\omega}_n}{2 \det{\tilde{\g}}} 
\frac{\partial \det{\tilde{\g}}}{\partial x^n} + \hspace{7cm} \nonumber \\
+ \frac{1}{2 \omega_n} \frac{\partial \g^{\alpha,\beta}}{\partial x^n} \omega_\alpha \omega_\beta - \frac{1}{2 \tilde{\omega}_n} \frac{\partial \tilde{\g}^{\alpha,\beta}}{\partial x^n} \omega_\alpha \omega_\beta  +\hspace{2cm} \label{E1}\\
   -2i (b^\beta \omega_\beta - \tilde{b}^\beta \omega_\beta) = O(\delta^{1/2}) \hspace{-1cm} \nonumber
\end{align}
in $L^2(V)$ for all  $\omega$'s as above. Setting $\omega'=0$,  we get
\[
\left|\left|\frac{1}{2 \det{\g}} 
\frac{\partial \det{\g}}{\partial x^n}- \frac{1}{2 \det{\tilde{\g}}} 
\frac{\partial \det{\tilde{\g}}}{\partial x^n}\right|\right|_{L^2(V)}  = O(\delta^{1/2}),
\]
and then since we can estimate the difference of the metric by (\ref{Estimates1L2}) we obtain
\[
\left|\left| \frac{\partial \det{\g}}{\partial x^n} - \frac{\partial \det{\tilde{\g}}}{\partial x^n} \right|\right|_{L^2(V)}  = O(\delta^{1/2}).
\]
This last equation together with (\ref{E1}) implies
\[
\left| \left|\frac{1}{2 \omega_n} \frac{\partial \g^{\alpha,\beta}}{\partial x^n} \omega_\alpha \omega_\beta - \frac{1}{2 \tilde{\omega}_n} \frac{\partial \tilde{\g}^{\alpha,\beta}}{\partial x^n} \omega_\alpha \omega_\beta -2i (b^\beta \omega_\beta - \tilde{b}^\beta \omega_\beta) \right|\right|_{L^2(V)} = O(\delta^{1/2}),
\]
since $\omega_n, \tilde{\omega}_n > \eta >0 $ and using again (\ref{PhaseEstimate}) we have 
\[
\left|\left| \left(\frac{\partial \g^{\alpha,\beta}}{\partial x^n} - \frac{\partial \tilde{\g}^{\alpha,\beta}}{\partial x^n}\right) \omega_\alpha \omega_\beta - (4\omega_n ib^\beta  - 4\omega_n i\tilde{b}^\beta) \omega_\beta\right|\right|_{L^2(V)} = O(\delta^{1/2}),
\]
Let us now take $\omega_n^2 = \tilde{\omega}^2_n = 1/2$, then 
\[
\left|\left| \left(\frac{\partial \g^{\alpha,\beta}}{\partial x^n} - \frac{\partial \tilde{\g}^{\alpha,\beta}}{\partial x^n}\right) \omega_\alpha \omega_\beta - (ib^\beta  -  i\tilde{b}^\beta) \omega_\beta \right|\right|_{L^2(V)} = O(\delta^{1/2}).
\]
and also $w'$ belongs to $\{\omega' \in T_{x'}\partial \Omega: |\omega '|_{\g'} =1 \}$ where $\g'$ is the induced metric of either $\g$ or $\tilde{\g}$ to $\partial \Omega$. Using now Lemma \ref{lemma2} and the fact that near the boundary $b_n - \tilde{b}_n = 0$, see (\ref{bnChange}), we have
\[
||b- \tilde{b}||_{L^2(V)} + \left|\left| \frac{\partial \g}{\partial x^n} -  \frac{\partial \tilde{\g}}{\partial x^n}  \right|\right|_{L^2(V)} \leq C\delta^{1/2}.
\]
Again by compactness we get,
\be \label{Estimates2L2}
||b- \tilde{b}||_{L^2(\partial \Omega)} + \left|\left| \frac{\partial \g}{\partial x^n} -  \frac{\partial \tilde{\g}}{\partial x^n}  \right|\right|_{L^2(\partial \Omega)} \leq C\delta^{1/2}.
\ee
As before using interpolation and Sobolev embeddings theorems
\be \label{Estimates2}
\Big\| \frac{\partial }{\partial x^n} (\g - \tilde \g)\Big\|_{C^m(\bo)}  +   \|  b- \tilde b\|_{C^m(\bo)} \le C \delta^{\mu/2}
\ee
\noindent
for any $m\ge0$ and $\mu<1$ as long as $k\gg1$.

To estimate the difference of the second normal derivatives of $\g$ and $\tilde{\g}$, first normal derivatives of $b$ and $\tilde{b}$ and the difference of $q$ and $\tilde{q}$ we use (\ref{DNMapBoundary}) up to the $\lambda^{-1}$ to get
\[ 
\left|\left| \frac{\partial A_1}{\partial x^n} - \frac{ \partial \tilde A_1}{\partial x^n} \right|\right|_{L^2(V)} \leq C\left( \lambda^2 \delta + \lambda\delta^{1/2} + \lambda^{-1}\right).
\]
Choose $\lambda = \delta^{-1/4}$ to obtain
\be\label{EstimatesSecondTerm}
\left|\left| \frac{\partial A_1}{\partial x^n} - \frac{ \partial \tilde A_1}{\partial x^n} \right|\right|_{L^2(V)} \leq C\delta^{-1/4}.
\ee
Using the equation $iLA_i = -(\partial_t^2 + P)A_0$ restricted to $I\times V$ we get
\[
2i\omega_n\frac{\partial A_1}{\partial x^n} = \frac{1}{2\det \g} \frac{\partial \det \g}{\partial x^n}\frac{\partial A_0}{\partial x^n} + \frac{\partial^2 A_0}{\partial^2 x^n} + q,
\]
this together with (\ref{PhaseEstimate}), (\ref{Estimates1L2}) and (\ref{EstimatesSecondTerm}) gives
\be
\left|\left|2\omega_n \frac{\partial^2 A_0}{\partial^2 x^n} - 2\tilde{\omega}_n\frac{\partial^2 \tilde{A}_0}{\partial^2 x^n} + 2\omega_n q- 2\tilde{\omega}_n\tilde{q} \right|\right|_{L^2(V)} = O(\delta^{1/4}).
\ee
Assuming $\omega_n, \tilde{\omega}_n > \eta >0$ for small $\eta$ and taking normal derivatives in (\ref{PrincPart}) and restricting it to $I\times V$ we have 
\beqno
2\omega_n\frac{\partial^2 A_0}{\partial^2 x^n}  & = & \frac{\omega_n}{2 \det{\g}} 
\frac{\partial^2 \det{\g}}{\partial^2 x^n} + \frac{\partial^3 \phi}{\partial^3 x^n} -2i \g^{\alpha\beta} \frac{\partial b_\alpha}{\partial x^n} \omega_\beta + R \\ 
 & = &  \frac{\omega_n}{2 \det{\g}} \frac{\partial^2 \det{\g}}{\partial^2 x^n} + \frac{1}{2\omega_n}\frac{\partial^2 \g^{\alpha\beta}}{\partial^2 x^n}\omega_\alpha \omega_\beta -2i \g^{\alpha\beta} \frac{\partial b_\alpha}{\partial x^n} \omega_\beta + R
\eeqno
where $R$ consists in terms that we can estimate by (\ref{Estimates1L2}) and (\ref{Estimates2L2}). Hence we get
\begin{align*}
\frac{\omega_n}{2 \det{\g}} 
\frac{\partial^2 \det{\g}}{\partial^2 x^n} - \frac{\tilde{\omega}_n}{2 \det{\tilde{\g}}} 
\frac{\partial^2 \det{\tilde{\g}}}{\partial^2 x^n} + 2\omega_nq-2\tilde{\omega}_n\tilde{q} +\hspace{7cm}\\
+ \frac{1}{2 \omega_n} \frac{\partial^2 \g^{\alpha,\beta}}{\partial^2 x^n} \omega_\alpha \omega_\beta - \frac{1}{2 \tilde{\omega}_n} \frac{\partial^2 \tilde{\g}^{\alpha,\beta}}{\partial^2 x^n} \omega_\alpha \omega_\beta  +\hspace{3cm} \\
   -2i \left(\g^{\alpha\beta} \frac{\partial b_\alpha}{\partial x^n} \omega_\beta - \tilde{\g}^{\alpha\beta} \frac{\partial \tilde{b}_\alpha}{\partial x^n} \omega_\beta \right) = O(\delta^{1/4}) \hspace{0cm}
\end{align*}
in $L^2(V)$. Again, setting $\omega' =0$ we have
\begin{equation} \label{Refe}
\left|\left|\frac{1}{2 \det{\g}} 
\frac{\partial^2 \det{\g}}{\partial^2 x^n} - \frac{1}{2 \det{\tilde{\g}}} 
\frac{\partial^2 \det{\tilde{\g}}}{\partial^2 x^n} + 2q-2\tilde{q} \right|\right|_{L^2(V)}= O(\delta^{1/4}). 
\end{equation}
Now since $\omega_n - \tilde{\omega}_n = O(\delta)$ then 
\[
\left|\left|\frac{1}{2 \omega_n} \left( \frac{\partial^2 \g^{\alpha,\beta}}{\partial^2 x^n} - \frac{\partial^2 \tilde{\g}^{\alpha,\beta}}{\partial^2 x^n} \right) \omega_\alpha \omega_\beta   -2i \left(\g^{\alpha\beta} \frac{\partial b_\alpha}{\partial x^n}- \tilde{\g}^{\alpha\beta} \frac{\partial \tilde{b}_\alpha}{\partial x^n}\right) \omega_\beta\right|\right|_{L^2(V)}  = O(\delta^{1/4}).
\]
We use similar reasoning as before. First we use Lemma  \ref{lemma2} and then compactness to get that $ \left|\left|\frac{\partial b}{\partial x^n}- \frac{\partial\tilde{b}}{\partial x^n}\right|\right|_{L^2(\partial \Omega)}$  and $ \left|\left| \frac{\partial^2 \g}{\partial^2 x^n} -  \frac{\partial^2 \tilde{\g}}{\partial^2 x^n}  \right|\right|_{L^2(\partial \Omega)}$ are $O(\delta^{1/4})$, this together with (\ref{Refe}) gives 
\be \label{Estimates3L2}
||q-\tilde{q}||_{L^2(\partial\Omega)} +  \left|\left|\frac{\partial b}{\partial x^n}- \frac{\partial\tilde{b}}{\partial x^n}\right|\right|_{L^2(\partial \Omega)} + \left|\left| \frac{\partial^2 \g}{\partial^2 x^n} -  \frac{\partial^2 \tilde{\g}}{\partial^2 x^n}  \right|\right|_{L^2(\partial \Omega)} \leq C\delta^{1/4}.
\ee
As before we get
\be \label{Estimates3}
||q-\tilde{q}||_{C^m(\partial\Omega)},  \left|\left|\frac{\partial b}{\partial x^n}- \frac{\partial\tilde{b}}{\partial x^n}\right|\right|_{C^m(\partial\Omega)}, \left|\left| \frac{\partial^2 \g}{\partial^2 x^n} -  \frac{\partial^2 \tilde{\g}}{\partial^2 x^n}  \right|\right|_{C^m(\partial\Omega)} \leq C\delta^{\mu/4}.
\ee
for $m>0$ and $\mu <1$. Proceeding by induction we prove the theorem.
\end{proof}

\begin{proof}[Proof of Lemma \ref{lemma2}]
First notice that by rescaling it is enough to proof the theorem for $r=1$. Let $\omega^1(x_0), \ldots,\omega^N(x_0)$ be like in Lemma 3.3 in \cite{MR2351370} related to $x_0$. Consider
\[
\omega^k(x) = \frac{\omega^k(x_0)}{|\omega^k(x_0)|_{\g}} \quad \mbox{for} \quad k = 1, \ldots, N. 
\]
By continuity of the metric, for any $\epsilon>0$, there exist a small neighborhood $U$ of $x_0$ were $|\omega^k(x_0)|_{\g} >1/2$ and $ |\omega^k(x) - \omega^k(x_0)|_{\g} < \epsilon$. Consider the linear transformation $L(x) : [h(x),b(x)]\mapsto [L(\omega^1(x), h(x) ,b(x)), \ldots, L(\omega^N(x), h(x),b(x))]$, notice that in Lemma 3.3  in \cite{MR2351370} the determination of $[h(x_0),b(x_0)]$ is done by inverting the linear transformation $L(x_0)$, whose inverse is also linear. Choose $\epsilon$ small enough so that $L(\omega^k(x),h(x), b(x))$ for $k = 1, \ldots, N$ are linear independent for all $x \in U$, were $U$ might be a smaller neighborhood of $x_0$, then we can take $C = \sup_{x\in U}||L^{-1}(x)||^{1/2}$. 
\end{proof}


\section{Interior Stability}

To obtain interior stability we first estimate the difference in the metrics and its derivatives following the argument in \cite{stefanov2005stable}. We use semigeodesical coordinates related to a point $z_0 \notin \bar{\Omega}$ but close enough so that simplicity assumption is still valid. We need to extend the geometrical optics solution \eqref{GeomOptSol} and reflect it at the boundary to satisfy initial conditions. We consider the phase function
\[
 \phi = x^n = d(x,z_0),
\] 
that is globally defined by simplicity assumption. The transport operator (\ref{TransEquatGeneralForm}) becomes 
\begin{equation} \label{TransEquatInterior}
L = 2\frac{\partial}{\partial t} + 2\frac{\partial}{\partial x^n}  +  \frac{1}{\sqrt{\det \g}} \frac{\partial \sqrt{\det \g} }{\partial x^n} - 2ib_n,  
\end{equation}
and can be solved explicitly. We get non-weighted integral transforms of $b$ and $q$ by  looking at the lower terms in the expansion of the solution. We then apply  H\"older stability results in \cite{stefanov2004stability} for vector fields and functions to estimate the difference of the covector fields and the potentials.


\subsection*{Modification of $\g$, $b$ and $q$ near the boundary}
\addcontentsline{toc}{subsection}{Modification of coefficients near the boundary}

For technical reasons we will need to make a $\delta$-small perturbation  of $\tilde{\g}, \tilde{b}$ and $\tilde{q}$ so that near the boundary they coincide with $\g, b$ and $q$. As in previous section, we denote $\delta = || \Lambda - \tilde{\Lambda}||_*$ and use notation in Remark \ref{TransformationDoesnotChangeAssumptions}.  

By Theorem \ref{BoundaryStability}, one has that for any $m >0$, there exist $ 0 < \mu < 1$ and $k\gg 1$, such that 					
\be\label{MetricApproximation}
\sup_{x\in\partial\Omega, |\gamma |\leq m} |\partial^\gamma(\g-\tilde{\g})|\leq C \delta^\mu.
\ee
Let $m>0$ be any integer and let $\chi \in C^\infty(\R)$, $\chi = 1$ for $t<1$ and $\chi = 0$ for $t>2$. Set
\be\label{ModificationMetric}
\tilde{\g}_1 = \tilde{\g} + \chi\left(\delta^{-1/M}\rho(x, \partial\Omega)\right)(\g-\tilde{\g}),  
\ee
where $M = 2m/\mu$
Using a finite Taylor expansion of $\g-\tilde{\g}$ up to $O((x^n)^M)$, where $x^n = \rho(x,\partial\Omega)$ and estimate (\ref{MetricApproximation}), we see that 
\begin{equation}\label{EstimateModificationOfMetric}
|| \tilde{\g}_1 - \tilde{\g}||_{C^m(\overline{\Omega})} \leq C\delta^{\mu - m/M} = C \delta^{\mu/2} 
\end{equation}

Is important to notice that $\tilde{\g}_1$ depends on $m$, but we will only need (\ref{ModificationMetric}) for large fixed $m$. In particular, the estimate above implies that $\tilde{\g}_1$ is also simple for $\delta\ll 1$. As before  without loss of generality we can assume that (\ref{MetricEstimates}) are still true for $\g$ and $\tilde{\g_1}$. We extend $\g$ and $\tilde{\g_1}$ in the same way as simple metrics in a neighborhood $\Omega_1 \supset \overline{\Omega}$. The advantage we have now is that
\be \label{MetricSameatBoundary}
\g = \tilde{\g_1} \mbox{ for } -1/C \leq \rho(x,\partial\Omega
) \leq \delta^{1/M}.
\ee
and hence by strictly convexity of the boundary there exist a constant $C$ such that 
\begin{equation} \label{BoundDistZeroAtNeigh}
(\rho_{\g} - \rho_{\tilde{\g}_1})|_{\partial \Omega 
\times \partial \Omega} = 0 \mbox{ if } \rho_\g(x,y) \leq C \delta^{1/(2M)},
\end{equation}
where $\rho_\g$ denotes the distance function with respect to $\g$. Moreover, using (\ref{EstimateModificationOfMetric}) we obtain 
\be \label{EstDistanceModif}
|\rho_{\tilde{\g}}(x,y) - \rho_{\tilde{\g}_1}(x,y)| \leq C\delta^{\mu/2}\;\; \forall x, y \in \overline{\Omega}.
\ee
We define similarly
\begin{align}\label{ModificationOfVectorAndPotencial}
\begin{split}
\tilde{b}_1&= \tilde{b} + \chi\left(\delta^{-1/M}\rho(x, \partial\Omega)\right)(b-\tilde{b}) \\
\tilde{q}_1&= \tilde{q} + \chi\left(\delta^{-1/M}\rho(x, \partial\Omega)\right)(q-\tilde{q}) 
\end{split}
\end{align}
and use Theorem \ref{BoundaryStability} to get that 
\begin{equation}\label{EstimateModificationOfVectorAndpotential}
|| \tilde{b}_1 - \tilde{b}||_{C^m(\overline{\Omega})} + || \tilde{q}_1 - \tilde{q}||_{C^m(\overline{\Omega})}  \leq C \delta^{1/(2M)} 
\end{equation}
and 
\begin{align}\label{VectAndPontSameNearBound}
\begin{split}
b = \tilde{b}_1 \mbox{ and } q = \tilde{q}_1 \mbox{ for } -1/C \leq \rho(x,\partial\Omega
) \leq \delta^{1/M}
\end{split}
\end{align}


\subsection*{Stable determination of the metric, the covector field and the potential}
\addcontentsline{toc}{subsection}{Stable determination of the metric, the covector field and the potential}

We first modify the mertic $\tilde{\g}, \tilde{b}$ and $\tilde{q}$ by $\tilde{\g}_1, \tilde{b}_1$ and $\tilde{q}_1$ as in (\ref{ModificationMetric}) and (\ref{ModificationOfVectorAndPotencial}). We use the same notation for the extensions. From now on objects below related to $\g$ are without tildes and those related to $\tilde{\g}_1$ are with a tilde above (no subscript 1). We proceed to the proof of the Main Theorem:

\begin{proof}[Proof of Theorem \ref{MainTheorem}]

Recall the notation $\delta = || \Lambda - \tilde{\Lambda}||_*$. It is enough to prove the proposition for $\delta \ll 1$. In what follows we denote by $\mu$, arbitrary constants $0<\mu<1$ that might decrease from step to step. Similarly we denote by $C$ various constants depending only in $\Omega$, $A$ and the choice of $k$ in Theorem \ref{MainTheorem} that could increase on each step.

Fix $x_0, y_0 \in \partial \Omega$. We will construct an oscillating solution related to $\g$ similarly to the one used to prove stability at the boundary, but in this case we want it to go all the way from $x_0$ to $y_0$. Consider the geodesic from $x_0$ to $y_0$, extended from $\Omega$ to $z_0 \in \Omega_1$ such that the geodesic segment $[z_0,x_0] \subset \Omega_1\backslash \Omega$, as illustrated by Figure \ref{fig:SolNearSemiGeodCoor}. Also, because of \eqref{MetricSameatBoundary}  we can assume that $\rho_\g(z_0,x_0) > 1/C >0$.

\begin{figure}[h] 
  \centering
  \includegraphics[height=2.7in,keepaspectratio]{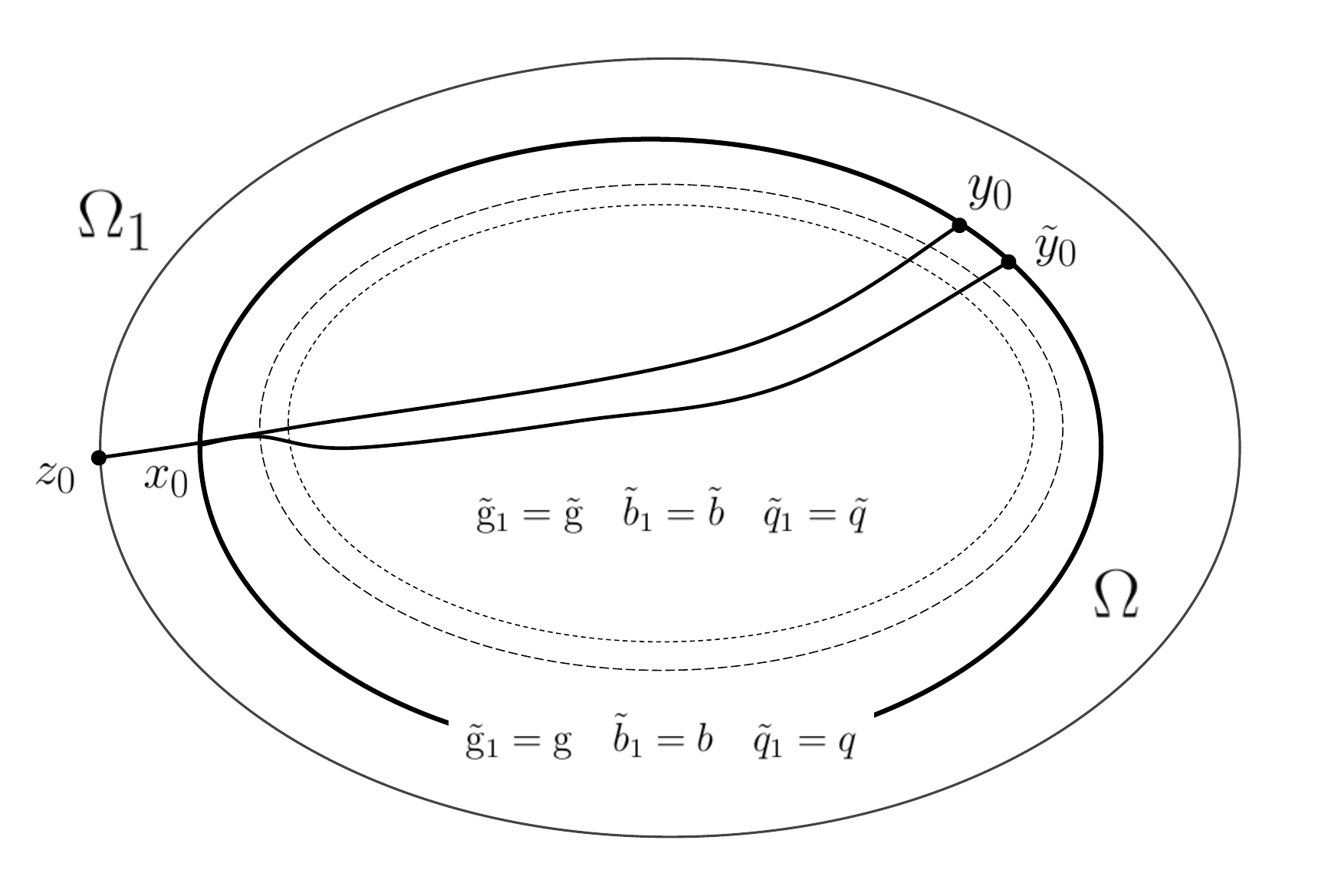}
  \caption{Solutions near semigeodesical coordinates.}
  \label{fig:SolNearSemiGeodCoor}
\end{figure}

Consider $(x',x^n)$ global semi-geodesic coordinates related to $z_0$,  where $x^n$ is the distance from $z_0$ to the point $p$ and $x'$ is a parametrization of the angular variable so that $x' = \mbox{const}$. are geodesics issued from $z_0$ with $x^n$ as the arc-lenght parameter, as in Lemma 4.2 in \cite{stefanov2005boundary}. We have that in these coordinates
\[
x^n = \rho_\g(p,z_0) \hspace*{1cm} \mbox{and} \hspace*{1cm} \g_{in} = \delta_{in}.
\]
in particula
Let $T > \mbox{diam}(\Omega)$, w.l.o.g. assume that $T- \mbox{diam}(\Omega)$ is small, see (\ref{IneqForDNMAp}). For $0< t_0 \ll 1$ define 
\begin{equation} \label{SetU}
U = \left\lbrace  (t,x) \in \R_+\times \partial \Omega; |t-t_0| + \rho(x,x_0) < \delta^{1/(2M)}/C_1 \right\rbrace 
\end{equation}
where $C_1 \gg 1$ will be specified later.  Choose a cut-off function $0 \leq \chi\leq  1, \chi \in C_0^\infty(\R_+\times \partial \Omega)$ such that $\sup \chi \subset U$, and $\chi =1$ in a set defined as $U$ but with $C_1$ replaced by $2C_1$. One can arrange that $|\partial_t\chi| + |\nabla \chi| \leq C\delta^{-1/(2M)}$.

Set $\phi(x) = \rho(x,z_0)$. Then, by simplicity assumption, since $z_0 \in \Omega_1$, we have that $\phi \in C^{k-1}(\overline{\Omega})$, and $\phi$ solves que eikonal equation
\[
\sum_{i,j=1}^\infty \g^{ij}\frac{\partial \phi}{\partial x^i}\frac{\partial \phi}{\partial x^j} = 1.
\]
Now we construct a solution $u$ of (\ref{InitProb}) for 
\[
f =  e^{-\ii \lambda(t-\phi)}\chi.
\]
We need to reflect the solution ones it reaches the boundary to get the zero boundary condition, so the solution $u$ is the sum of the incident wave and the reflected wave, $u = u^{\inc} + u^{\refl}$ with 

\[
u^{\inc}(t,x;\lambda) = e^{-\ii \lambda (t- \phi(x))}(A^{\inc}_0(t,x) + \lambda^{-1}A^{\inc}_1(t,x) +  R^{\inc}(t,x;\lambda)),
\]
\[
u^{\refl}(t,x;\lambda) = e^{-\ii \lambda (t- \hat{\phi}(x))}(A^{\refl}_0(t,x) + \lambda^{-1}A^{\refl}_1(t,x) +  R^{\refl}(t,x;\lambda)).
\]
where
\be\label{C2Estimates}
||R^{\inc} ||_{C^1([0,T];L^2(\Omega))} + ||R^{\refl}||_{C^1([0,T];L^2(\Omega))}\leq \frac{C}{\lambda^2},
\ee
This solutions satisfy the transport equations
\be \label{TransForInitWave}
LA^{\inc}_0 = 0, \hspace{0.2cm}   A^{\inc}_0|_U = \chi;\hspace{0.5cm}  
\ii LA^{\inc}_1 = -(\partial_t^2 +P)A^{\inc}_{0}, \hspace{0.2cm} A^{\inc}_j|_U=0 
\ee
\be \label{TranForRefWave}
LA^{\refl}_0 = 0, \hspace{0.2cm}   A^{\refl}_0|_V = -A^{\inc}_0|_V;\hspace{0.5cm}  
\ii LA^{\refl}_1 = -(\partial_t^2 +P)A^{\refl}_{0}, \hspace{0.2cm} A^{\refl}_1|_V= -A^{\inc}_1|_V 
\ee
where $L$ is as in \eqref{TransEquatInterior} and  $V \subset \R_+ \times \partial \Omega$ is the image of  $U$  under translations by all geodesics issued from $z_0$ and passing through  $U$.  The phase function $\hat{\phi}$ still solves the eikonal equation with boundary condition $\hat{\phi}|_V = \phi$ and is unique solution with gradient pointing towards the interior of $\Omega$ (the opposite solution in $\phi$). All these boundary conditions are assumed to be extended as zero in the rest of the boundary. 
 
We can solve the transport equations (\ref{TransForInitWave}) and (\ref{TranForRefWave}) in a neighborhood of the geodesic connecting $x_0$ and $y_0$ of size $O(\delta^{1/(2M)})$, and by simplicity assumption, this solutions can be extended all the way to $V$. If $C_1$ in (\ref{SetU}) is large enough, then $U$ and $V$ are disjoint sets
 
   Because of the strict convexity of $\partial \Omega$, each component is of size $O(\delta^{1/(2M)})$, at a distance bounded from below by the same quantity by assumption. Denote by $B(y,r)$ the ball centered at $y$ with radius $r$. Then $V$ contains the set $V_0 = V\cap B(y_0, \delta^{1/(2M)}/C^0)$, such that on $V_0$, we have $A^{\inc}_0 \geq 1/C >0$. Above, $C^0$ is chosen so that $V_0$ is contained in the translation of the set $\{ \chi =1 \}$ under geodesics issued from $z_0$.

By estimating $A_0\nabla \phi$ on $U$, and using the fact that $|\partial_t\chi| + |\nabla \chi| \leq C\delta^{-1/(2M)}$ we get
\begin{equation} \label{EstofSource}
||f||^2_{H^1([0,T]\times\partial\Omega)} \leq C\delta^{n/(2M)}\lambda^2 + C(\delta) 
\end{equation}
where the first constant is independent of $\delta$ (but it depends on $A$ and $\epsilon_0$ in (\ref{MetricEstimates})). Finally on $V_0$, we have 
\begin{equation}\label{DNMapInter1}
\Lambda_Pf =   e^{-\ii \lambda(t- \phi)}  \left( -2\ii\frac{\partial \phi}{\partial \nu} A^{\inc}_0 \lambda  + A^{\inc}_1  + \frac{\partial A_0 }{\partial \nu} \right)+ O_\delta(1/\lambda)
\end{equation}
in $L^2(V_0)$, where $A_0 = A^{\inc}_0 + A^{\refl}_0$. 

We construct a similarWe can s solution $\tilde{u}$ related to $\tilde{\g}$. First we construct a phase function $\tilde{\phi}$ as $\tilde{\phi} = \tilde{\rho}(x,z_0)$. Because of (\ref{MetricSameatBoundary}), $\tilde{\phi}$  solves the eikonal equation 
\begin{equation}
\sum_{i,j=1}^n \tilde{\g}^{ij}\dfrac{\partial \tilde{\phi}}{\partial x^i} \dfrac{\partial \tilde{\phi}}{\partial x^j} = 1, \hspace{0.5cm} \tilde{\phi}|_{U} = \phi.
\end{equation}
The other properties of $\tilde{u}$ are similar to those of $u$. Let $\tilde{V}_0$ be defined as above, but associated to $\tilde{\g}$, in $\tilde{V_0}$ we obtain an expression for $\Lambda_{\tilde{P}}\tilde{f}$ similar to one obtained in \eqref{DNMapInter1}.

Next we divide the proof in three parts. We first prove stability for the metric, then for the covector field and finally for the potential. In each step we use higher order terms in the asymptotic solution constructed above. 

\noindent
\textit{(a) Metric Stability:} We follow \cite{stefanov2005stable} and use boundary rigidity estimates in \cite{stefanov2005boundary}. Some simplifications to the argument are presented. Notice that
\begin{equation}\label{LowerBoundAmplitudeDN}
\left|\left| \frac{\partial \phi}{\partial x^n} A^{\inc}_0\right|\right|^2_{L^2(V_0)} \geq \delta^{n/(2M)}/C
\end{equation}
because $|V_0|\geq \delta^{n/(2M)}/C$.

We  claim that $V_0 \cap \tilde{V}_0 \neq \emptyset$. Arguing by contradiction, since $V_0$ and $\tilde{V}_0$ do not intersect and by (\ref{EstofSource}) we obtain
\[
\lambda^2\delta^{n/(2M)}/C - C(\delta) \leq ||\Lambda_{P}f - \Lambda_{\tilde{P}}f||^2_{L^2[0,T]\times \partial \Omega} \leq C \lambda^2\delta^{1+n/(2M)}  + C(\delta). 
\]
Dividing last inequality by $\lambda^2$ and taking $\lambda \to \infty$ we get a contradiction. Hence $ q \in V_0 \cap \tilde{V}_0$, using simplicity assumption and triangular inequality  we get
\be \label{DistanceEstimate1}
| \rho^2_\g(x_0,y_0) - \rho^2_{\tilde{\g}}(x_0,y_0)| = O(\delta^{\mu})
\ee
with some $0<\mu<1$ uniformly w.r.t. $x_0$ and $y_0$. In view of \eqref{BoundDistZeroAtNeigh} we can assume that $\rho_g(x_0,y_0) > C \delta^\mu $, for $0< \mu <1$. Using \eqref{DistanceEstimate1} and  (\ref{EstDistanceModif}) we have
\[
|| \rho_{\g} - \rho_{\tilde{\g}} ||_{C(\partial \Omega \times \partial \Omega)} \leq C \delta^\mu \mbox{ for } 0 < \mu < 1.
\]
We then apply estimate in Theorem 1.8 in \cite{stefanov2005boundary} to get
\begin{equation}\label{EstimateInteriorMetrics}
||\g-\varphi^* \tilde{\g}||_{C^2(\overline{\Omega})} \leq C \delta^\mu \mbox{ for } 0 < \mu <1.
\end{equation}
where $\varphi: \Omega \to \Omega$ is some diffeomorphism fixing the boundary. From now on we pull-back the the metric $\tilde{\g}$, the covector $\tilde{b}$ and the potential $\tilde{q}$ by the diffeomorphism $\varphi$. We denote them again by $\tilde{\g}$, $\tilde{b}$ and $\tilde{q}$. Notice that by modifing as in \eqref{ModificationOfVectorAndPotencial} we can assume that $ b = \tilde{b}$ and $q = \tilde{q}$ in a $O(\delta^{1/M})$ neighborhood of the boundary as in \eqref{VectAndPontSameNearBound}.

\noindent
\textit{(b) Magnetic Field Stability:} For this part we will use the stability of the principal part in the solution constructed in $(a)$ and stability of the 1-tensor geodesic X-ray transform. We will use sharp estimates of the 1-tensor X-ray transform obtained in \cite{stefanov2004stability}. Stability for this 1-tensor geodesic X-ray transform was previously know, see \cite{sharafutdinov1994integral} and references in there. From \eqref{EstofSource} we know
\[ 
||f||^2_{H^1([0,T]\times\partial\Omega)} \leq C\lambda^2 + C(\delta) 
\]
where the first constant is independent of $\delta$, but it depends on $A$ and $\epsilon_0$ in (\ref{MetricEstimates}). We also have 
\[
||\Lambda_{P}f - \Lambda_{\tilde{P}}f||^2_{L^2[0,T]\times \partial \Omega} \leq C \lambda^2\delta  + C(\delta).
\]
Using these two last equations, (\ref{DNMapInter1}) and part $(a)$  we obtain
\[
||A^{\inc}_0 - \tilde{A}^{\inc}_0||_{L^2([0,T] \times V_0\cap \tilde{V}_0)} \leq C\delta + C(\delta)/\lambda, 
\]
taking $\lambda \to \infty$ we get
\begin{equation} \label{EstimateAmplitudInterior}
||A^{\inc}_0 - \tilde{A}^{\inc}_0||_{L^2([0,T] \times V_0\cap \tilde{V}_0)} \leq C \delta 
\end{equation}
In this coordinate system, remember $\phi = x^n$, the the transport equation (\ref{TransEquatGeneralForm}) becomes
\begin{equation} \label{TransportEquationForPhaseFunctionInGlobalCoord}
\left(\frac{\partial}{\partial t} + \frac{\partial}{\partial x^n} + \frac{1}{2\sqrt{\det \g}} \frac{\partial \sqrt{\det \g} }{\partial x^n} - ib_n \right)A^{\inc}_0 = 0
\end{equation}
with same initial conditions as in (\ref{TransForInitWave}) and $b = b_i\mbox{d}x^i$. Using the method of characteristics or a change of variables we can  compute explicitly this solution and get
\[
A^{\inc}_0(t,x) = \chi((x',0),t-x^n) \left(\tfrac{\det \g(x',0)}{\det \g(x)}\right)^{1/4} \exp \left(\ii  \int_{0}^{x^n} b_n(x', s) ds \right).
\]
If we stay near $t=x^n$, then $\chi =1$ so 
\begin{equation} \label{AmplitudeInterior}
A^{\inc}_0(t,x) = \left(\tfrac{\det \g(x',0)}{\det \g(x)}\right)^{1/4} \exp \left(\ii  \int_{0}^{x^n} b_n(x', s) ds \right).
\end{equation}
Using (\ref{EstimateAmplitudInterior}), and (\ref{EstimateInteriorMetrics}) we get that 
\be
\left|\left|  \exp \left( \ii  \int_{0}^{x^n} b_n(x',s) - \tilde{b}_n (x',s) ds \right) -1  \right|\right|_{L^2([0,T] \times V_0\cap \tilde{V}_0)}   \leq C\delta^\mu. 
\ee
Remember that we modified the covector field to that $b-\tilde{b} = 0$ in a neighborhood of the boundary containing $\Omega_1\setminus\Omega$, then $b-\tilde{b}$ belongs to $C^1(\overline{\Omega})$ and by  (\ref{AprioriCond}) and (\ref{EstimateModificationOfVectorAndpotential}) we have $||b - \tilde{b}||_{C^1(\bar{\Omega})} \leq C (\epsilon_0 + \delta^{\mu/2})$ (i.e., arbitrarly small). Then we can use a Taylor expansion of $e^{ix}$ near zero to get
\be\label{IntegralBOverGeodesic}
\left|\left|  \int_{0}^{x^n} b_n(x',s) - \tilde{b}_n(x',s) ds \right|\right|_{L^2([0,T] \times V_0\cap \tilde{V}_0)}   \leq C\delta^\mu. 
\ee
This is the coordinate representation of the X-ray transform along the geodesic starting at $z_0$ and going all the way to $y_0$. Until know we had a fixed coordinate system associated to $z_0$. Since all constants $C$ are uniform with respect to $x_0$ and $y_0$ We can then shoot rays in all directions and we can move $z_0$ around all $\partial \Omega_1$. We note that since we modify the covector field near the boundary of $\Omega$ there are no non-zero tangential rays that are being integrated over. Hence, since $I_\g: L^2(\Omega_1) \to L^2(\Gamma_-(\Omega_1);d\mu)$ is bounded  by \cite{sharafutdinov1994integral} we have 
\begin{equation}
|| I_\g(b - \tilde{b})||_{L^2(\Gamma_-(\Omega_1);d\mu)} \leq C \delta^{\mu}.
\end{equation}
Using  interpolation and the compactness assumption of the metric and the potential we can estimate 
\[
|| I_\g(b - \tilde{b})||_{C^1(\Gamma_-(\Omega_1);d\mu)} \leq C \delta^{\mu}.
\]
By Theorem $4$ in \cite{stefanov2004stability} we know

\begin{equation}
 ||b-(\tilde{b} - \mbox{d}\theta) ||_{L^2(\Omega)} \leq C ||I^*_gI_\g (b-\tilde{b})||_{H^1(\Omega)}
\end{equation}
for $\theta \in C^\infty(\Omega,\R)$ such that $\theta = 0$ on $\Omega_1 \setminus\mbox{int}(\Omega)$ and with 
\[
(I^*_\g \psi(x))^i = \int_{|\xi|_\g=1} \xi^i \psi(\gamma_{x,\xi}) d\xi,
\]
were $\gamma_{x,\xi}$ denotes the maximal geodesic in $\Omega_1$ that passes through $x$ with codirection $\xi$. Hence 

\begin{equation}
 ||b-(\tilde{b} - \mbox{d}\theta) ||_{L^2(\Omega_1)} \leq C\delta^{\mu}
\end{equation}
for some $0<\mu<1$. Now by (\ref{EstimateModificationOfVectorAndpotential}) and since $\tilde{b} = \tilde{b}_1$ on $\Omega_1\setminus \Omega$ then 
\begin{equation}
 ||b-(\tilde{b} - \mbox{d}\theta) ||_{L^2(\Omega)} \leq C\delta^{\mu}.
\end{equation}
By interpolation and compactness we get 
\begin{equation}
 ||b-(\tilde{b} - \mbox{d}\theta) ||_{C^1(\Omega)} \leq C\delta^{\mu}.
\end{equation}
From now on change $\tilde{b}$ by $\tilde{b} - \mbox{d}\theta$, but we keep the same notation $\tilde{b}$.

\noindent
\textit{(b) Potential Stability:} Finally for the potential we use the next term in the expansion of the previous solution and stability estimates for the geodesic X-ray transform of functions in \cite{stefanov2004stability}. We have that the DN-map is given by \eqref{DNMapInter1} in $L^{2}([0,T] \times V_0)$ and similarly for $\Lambda_{\tilde{P}}$. Notice that since the metrics $\g$ and $\tilde{\g}$ are equal in a neighborhood of the boundary and because simplicity assumption, then all rays can be taken transversal to the boundary. Hence 
 \[
 \frac{\partial}{\partial \nu} = a(x) \frac{\partial}{\partial x^n} + c(x)\frac{\partial}{\partial \nu^{\perp}}
 \]
with $|a(x)| > 1/C$  in $V_0 \cap \tilde{V}_0$ , where $\partial_{\nu^{\perp}}$ is tangential to the boundary. Since we can estimate tangential derivatives we get 
 \[
 -2\ii \frac{\partial \phi }{\partial \nu}(A^{\inc}_1 - \tilde{A}^{\inc}_1) = \Lambda_Pf-\Lambda_{\tilde{P}}\tilde{f} + 2\ii a(x)\lambda (A^{\inc}_0 - \tilde{A}^{\inc}_0) - a(x)\frac{\partial}{\partial x^n}(A_0 - \tilde{A}_0)  + O(\lambda^{-1}) + O(\delta^\mu)
\]
in $L^2([0,T]\times V_0 \cap \tilde{V}_0)$. Now, using the explicit form of the amplitude (\ref{AmplitudeInterior}) to estimmate $A_0 - \tilde{A}_0$ we obtain
\[
\left|\left| - \ii  \frac{\partial \phi}{\partial \nu} (A^{\inc}_1 -\tilde{A}^{\inc}_1)\right|\right|_{L^2([0,T]\times V_0 \cap \tilde{V}_0)} \leq  C (\lambda\delta + 1/\lambda) + O(\delta^\mu),
\]  
taking $\lambda = \delta^{-1/2}$ in the previous inequality we get
\begin{equation} \label{InteriorEstimateForA1}
\left|\left| A^{\inc}_1 - \tilde{A}^{\inc}_1\right| \right|_{L^2([0,T]\times V_0 \cap \tilde{V}_0)} \leq C \delta^{\mu}.
\end{equation}
In our coordinates the transport equations \eqref{TransForInitWave} becomes 
\begin{equation}
\left(\frac{\partial}{\partial t} + \frac{\partial}{\partial x^n} + \frac{1}{2\sqrt{\det \g}} \frac{\partial \sqrt{\det \g}}{\partial x^n} - ib_n \right)A^{\inc}_1 = (\partial^2_t + P)A^{\inc}_0,
\end{equation}
where by (\ref{CoeficientsTransformationsForSeflfadjointness})  
\[
(\partial^2_t + P)A^{\inc}_0 = \partial^2_tA^{\inc}_0 -\triangle_gA^{\inc}_0 - qA^{\inc}_0 - |b|^2_gA^{\inc}_0 - \ii\left( 2 \langle b, \mbox{d}A^{\inc}_0\rangle_\g + (\mbox{Div}b)A^{\inc}_0\right).
\]
As before, solving by the method of characteristics, near $t = x^n$ we get 
\begin{align*}
A^{\inc}_1(x,t) & = \tfrac{1}{\eta(x)}  \int_0^{x^n} [ \eta(x',s) \triangle_gA^{\inc}_0 -(\det \g(x',0')^{1/2}q   - (\det \g(x',0))^{1/4}|b|_\g   \\& \hspace*{3cm}+  \ii  2\eta(x',s) \left( 2 \langle b, \mbox{d}A^{\inc}_0\rangle_\g + (\mbox{Div}b)A^{\inc}_0 \right)] ds  
\end{align*}
where 
\[
\eta(x',x^n) = (\det \g(x))^{1/4} \exp \left( -\ii \int_0^{x^n} b(x',s)ds\right)
\]
Now by (\ref{InteriorEstimateForA1}) and previous estimates we have
\[
\left|\left|\tfrac{(\det \g(x',0))^{1/4}}{\eta(x)} \int_0^{x^n}q_n(s) ds - \tfrac{(\det \tilde{\g}(x',0))^{1/4}}{\tilde{\eta}(x)} \int_0^{x^n}\tilde{q}_n(s) ds\right|\right|_{L^2([0,T]\times V_0 \cap \tilde{V}_0)} \leq C\delta^{\mu},
\]
since we have estimates for the metrics and the covector field this implies
\[
\left|\left|  \int_{0}^{x^n} q(x',s) - \tilde{q}(x', s) ds \right|\right|_{L^2([0,T]\times V_0 \cap \tilde{V}_0)}   \leq C\delta^\mu.
\]
Comparing this with (\ref{IntegralBOverGeodesic}), we argue as before and we obtain an invariant representation of the previous inequality
\[
|| I_\g(q - \tilde{q})||_{L^2(\Gamma_-(\Omega_1);d\mu)} \leq C \delta^{\mu}.
\]
Using  interpolation and the compactness assumption of the metric and the potential we get estimate  (remember we changed the potential so that $q = \tilde{q}$ near the boundary)
\[
|| I_\g(q - \tilde{q})||_{C^1(\Gamma_-(\Omega_1);d\mu)} \leq C \delta^{\mu}.
\]
We then apply a stability estimate for the geodesic X-ray transform as in Theorem $3$ in \cite{stefanov2004stability} and interpolation to obtain
\[
||q- \tilde{q}||_{C(\Omega)}\leq C\delta^{\mu}
\]
for some $0<\mu<1$.
\end{proof} 

\bibliographystyle{amsplain.bst}
\bibliography{MathBiblio}

\end{document}